\documentclass[12pt,reqno,a4paper]{amsart}
\usepackage{
    amsmath,  amsfonts, amssymb,  amsthm,   amscd,
    gensymb,  graphicx, comment,  etoolbox, url,
    booktabs, stackrel, mathtools,enumitem, mathdots,  microtype, lmodern,    mathrsfs, graphicx, tikz,  longtable,tabularx, float, tikz, pst-node, tikz-cd, multirow, tabularx, amscd,  bm, array, makecell, diagbox, booktabs,ragged2e, caption, subcaption }
\usepackage{makecell,slashbox}
\usepackage{algorithm}
\usepackage{algpseudocode}
\usepackage{esvect}

\MakeRobust{\vv}
\usepackage{xcolor}
\usepackage[utf8]{inputenc}
\usepackage{microtype, fullpage, wrapfig,textcomp,mathrsfs,csquotes,fbb}
\usepackage[colorlinks=true, linkcolor=blue, citecolor=blue, urlcolor=blue, breaklinks=true]{hyperref}
\usepackage[capitalise]{cleveref}
\usepackage{todonotes}
\usetikzlibrary{positioning}
\usetikzlibrary{shapes,arrows.meta,calc}
\usetikzlibrary{arrows}

\newtheorem{theorem}{Theorem}[section]

\newtheorem{corollary}[theorem] {Corollary}
\newtheorem{definition}[theorem]{Definition}
\newtheorem{example}[theorem]{Example}
\newtheorem{lemma}[theorem]{Lemma}

\newtheorem{proposition}[theorem]{Proposition}
\newtheorem{remark}[theorem]{Remark}

\newcommand{\ct}{{\mathrm{cat}}}
\newcommand{\cat}{\vv{\mathrm{cat}}}

\newcommand{\TC}{\mathrm{TC}}

\newcolumntype{x}[1]{>{\centering\arraybackslash}p{#1}}

\begin{document}
%\title[]{On the Sequential directed (parametrized) topological complexity}
\title[]{On the Sequential topological complexity of directed (parametrized) motion planning algorithms}
\author[N. Daundkar]{Navnath Daundkar}
\address{Department of Mathematics, Indian Institute of Technology Madras, Chennai, India.}
\email{navnath@iitm.ac.in}
\author[A. Sarkar ]{Abhishek Sarkar}
\address{School of Computing, MIT Art, Design and Technology University, Loni Kalbhor, Pune, India.}
\email{abhishek.sarkar@mituniversity.edu.in}
\author[A. Sarkar]{Ankur Sarkar}
\address{Indian Statistical Institute, Theoretical Statistics and Mathematics Unit, Kolkata, India.} 
\email{ankurimsc@gmail.com}
%CIT Campus, Taramani, 600113
  \subjclass [2020] {Primary : 55M30, 55S40, 55P99}  %{; Secondary : }}
	\keywords{Sequential (parametrized) topological complexity, directed topological complexity, directed LS category, sequential directed (parametrized) topological complexity}

\begin{abstract} 
We introduce sequential analogues of directed (parametrized) topological complexity, in the context of motion planning problems requiring a system to traverse a prescribed sequence of intermediate states while respecting directed dynamics and varying external parameters.
We develop their basic theory, establish fundamental properties, and compute them for several classes of examples. 
Our computations show, in particular, that distinct directed structures on the same underlying space can have different values of this invariant.
%Our computations show, in particular, that distinct directed structures on the same underlying space can exhibit different sequential topological complexities.
\end{abstract}
\maketitle

\section{Introduction} 
In motion-planning theory, topological spaces often serve as configuration spaces encoding all allowable states of a mechanical or robotic system. Recent advances have focused on invariants that measure the complexity of navigating these spaces, particularly in settings involving multiple stages or external parameters. 
In this context, sequential and parametrized invariants have become central, notably the parametrized topological complexity introduced by Cohen--Farber--Weinberger \cite{C-F-W} and extended by Farber--Paul \cite{Farber-Paul1} to the sequential setting.
These invariants model algorithms that guide a system through a prescribed sequence of states under varying external conditions, encoded by a fibration $p\colon E\to B$ whose fibre $X_b = p^{-1}(b)$ at $b \in B$ represents the system states.
The invariant $\TC_n[p\colon E\to B]$ measures the complexity of such universal motion-planning algorithms in the parametrized sequential setting.
These developments build on foundational invariants, notably the Lusternik--Schnirelmann (LS) category $\ct(X)$ \cite{LScat}, Farber's topological complexity $\TC(X)$ \cite{FarberTC}, and Rudyak's higher (or sequential) topological complexity $\TC_r(X)$ \cite{RUD2010} for a path-connected topological space $X$. These invariants appear in applied algebraic topology, critical point theory, robot motion planning (see \cite{Farber_Book}), and other areas.

While classical motion-planning invariants assume reversible paths, many real-world systems exhibit inherent directionality due to time or irreversible constraints   (see the works of Grandis \cite{grandis2002directed,grandis2003directed,MG2009} and \cite{fajstrup2016directed}). 
%While classical motion-planning invariants assume reversible paths, many real-world systems exhibit inherent directionality due to temporal or irreversible constraints, motivating the development of directed topology by Grandis (see \cite{grandis2002directed,grandis2003directed,grandis2009directed}) and subsequent work by Fajstrup \emph{et al.} \cite{fajstrup2016directed}.
%In these settings, mutual reachability is not guaranteed. 
This motivates the study of directed spaces, which are topological spaces equipped with a distinguished collection of admissible paths encoding these motions. Motion planning in directed settings is significantly more complex, requiring the continuous selection of admissible paths strictly between reachable states. Since the classical topological complexity and its variants fail to capture these directed constraints, this leads to the development of directed analogues of these invariants.
The notion of \emph{directed topological complexity}, introduced by Goubault, et al. in \cite{goubault2017directed,EGMFAS2020}, extends Farber's topological complexity to motion planning problems.
%with directional constraints, motivated by applications in concurrency theory and directed dynamical systems    
For a directed (or, $d$-) space $(X,dX)$, restricting the free path fibration to directed (or, $d$-) paths yields
$\vv{\pi}\colon dX \rightarrow \Gamma_X,$
where $\Gamma_X \subseteq X \times X$ consists of pairs of points connected by directed paths. The integer $\vv{\TC}(X)$ is defined as the least number of Euclidean neighbourhood retracts (ENRs) covering $\Gamma_X$, each admitting a continuous local section of $\vv{\pi}$. 

In many physical systems, the configurations under consideration are constrained not only by external parameters but also by an inherent irreversibility in their allowed motions. 
In such settings, paths cannot simply be reversed, meaning the classical fiberwise path space is no longer adequate and must be replaced by a space of directed fiberwise paths.
To address this, the \emph{directed parametrized topological complexity} of a $d$-fibration $p\colon E\to B$ was introduced in~\cite{SDNDAS2025} as an extension of directed topological complexity to the parametrized setting.
%directed topological complexity and parametrized topological complexity were combined to introduce the \emph{directed parametrized topological complexity} of a d-fibration $p\colon E \to B$. %This invariant measures the complexity of constructing continuous directed motion planners compatible with the fibration structure. 
For a $d$-fibration $p\colon E \to B$, restricting the endpoint fibration to the space of directed fibrewise paths gives
$\vv{\Pi}\colon dE_B \rightarrow \Gamma_{E,B},$
and the invariant $\vv{\TC}[p\colon E \to B]$ is defined as the least number of ENRs covering $\Gamma_{E,B}$ such that each admits a continuous section of $\vv{\Pi}$. 

In the present work, first we extend the framework of directed topological complexity 
to the setting of \emph{sequential motion planning} (see Section \ref{Seq directed TC}).
Consequently, we introduce a sequential analogue of the directed parametrized topological complexity (see Section \ref{seq directed parametrized TC}).
For that we consider the $d$-maps associated with a $d$-space $X$ and a $d$-fibration $p\colon E \rightarrow B$
$$\vv{\pi}_n\colon dX\to \Gamma_{X}^n, ~~\text{and}~~ \vv{\Pi}_n\colon dE_B \rightarrow \Gamma_{E,B}^n$$ and study their local sections on ENRs' analogous to the earlier cases, where $\Gamma_{X}^n, ~\Gamma_{E,B}^n$ are the associated spaces of end points.
The main objective is to study motion planning problems in which a system is required to execute a prescribed sequence of intermediate tasks or states while respecting both directional constraints and external parameters. Motivated by applications in robotics, autonomous systems, and concurrent processes, we study the \emph{sequential directed topological complexity} and the \emph{sequential directed parametrized topological complexity} as directed analogues of Rudyak's higher topological complexity or sequential analogue of Goubault's directed topological complexity and their associated (directed or sequential) parametrized analogues. These invariants measure the complexity of constructing continuous directed motion planners for multiphase tasks, where admissible motions must follow the directed structure and remain compatible with the parametrized environment. 
We investigate their fundamental properties, invariance under sequential dihomotopy equivalence, and their relationships with sequential topological complexity and sequential parametrized topological complexity, respectively. We also present several examples and applications arising from directed configuration spaces and sequential robot motion planning problems. 

%One of the major contributions of this paper is the introduction of the sequential analogue of basic dihomotopy equivalence (see Definition~\ref{n-basic dihomotopy equivalence}) and the establishment of the corresponding dihomotopy invariance property (see Proposition~\ref{n-dihmotopy invariance}). The motivation for introducing this notion comes from the fact that sequential directed topological complexity is not homotopy invariant. Under certain assumptions, we obtain the sequential analogue (see Proposition~\ref{compare cat and TC}) of \cite[Proposition 3.6]{SDNDAS2025}, namely,

We show that sequential directed topological complexity is not homotopy invariant. This motivates the introduction of the sequential analogue of basic dihomotopy equivalence (see Definition \ref{n-basic dihomotopy equivalence}), which requires not only a homotopy equivalence between the underlying $d$-spaces but also compatible homotopy equivalences between the corresponding sequential directed path spaces. We then establish the associated dihomotopy invariance property (see Proposition \ref{n-dihmotopy invariance}). Under certain assumptions, we also obtain the sequential analogue (see Proposition \ref{Properties_of_higher_TC}) of \cite[Proposition 3.6]{SDNDAS2025}, namely, $\cat(X^{n-1})\leq \vv{\TC}_n(X)$, where $\cat(Y)$ denoted the directed LS category of the directed space $Y$.

We also introduce the sequential analogue of regular $d$-spaces, called $n$-regular $d$-spaces, and compare sequential directed topological complexity with the usual sequential topological complexity (see Proposition~\ref{comparision} and Proposition~\ref{Comparision_TC_DirectedTC}). In particular, if the $d$-space $X$ is strongly connected, then $\TC_n(X)\leq \vv{\TC}_n(X)$.
We then investigate sequential directed topological complexity for products of $d$-spaces (see Section~\ref{product d-spaces}). Finally, we prove that for a contractible $d$-space $X$, the map $\vv{\pi}_n\colon dX\to X^n$ admits a continuous global section if and only if $X$ is $n$-basic dihomotopy equivalent to a point (see Proposition~\ref{Important_Result}).

We further introduce the notion of fibrewise $n$-basic dihomotopy equivalence (see Section \ref{fibrewise n-dihomotopy equivalence}) and show that for $d$-fibrations $p\colon E \rightarrow B$ and $p':E' \rightarrow B'$  with such equivalence, we obtain 
$\vv{\TC}_n[p\colon E \rightarrow B]= \vv{\TC}_n[p'\colon E '\rightarrow B'].$ 
Moreover, we show that for a $d$-fibration $p\colon E\to B$ with fibre $F$ being either a contractible $d$-space or having an initial point, $F$ is $n$-basic dihomotopy equivalent to a point if and only if $\vv{\TC}_n[p\colon E\to B]=1$ (see Theorem \ref{dicontractible fibre}).
In Section \ref{Sequential analogue of regular d-fibrations}, we define $n$-(sequential analogue of) regular $d$-fibrations
and  prove the product inequality for the sequential directed parametrized topological complexity associated with two $n$-regular $d$-fibrations (see Proposition \ref{product inequality_parametrized}).
 For a particular directed structure introduced in \cite[Definition 5.1]{SDNDAS2025}, we show that $\vv{\TC}_n[p\colon E\to B]=\TC_n[p\colon E\to B]$,
holds for a fibration $p\colon E\to B$ consists with a $d$-structure on $B$.

We consider several examples of topological spaces equipped with their standard directed structures and compute their sequential directed topological complexities, as well as their sequential directed parametrized topological complexities. For example in Proposition \ref{directed_circle}, we consider a directed structure on the circle $\mathbb{S}^1$, denoted by $\vv{\mathbb{S}^1}$, such that $d\mathbb{S}^1 \neq P\mathbb{S}^1$, but we obtain $\vv{\TC}_n(\vv{\mathbb{S}^1})=n=\TC_n(\mathbb{S}^1).$
We compute $\vv{\TC}_n(\mathbb{O}^1)$ for the directed loop $\mathbb{O}^1$ (see Proposition \ref{TC_O^1}).
We compute $\vv{\TC}_n(-)$ of a torus with different $d$-structures (see Corollary \ref{directed torus} and Corollary \ref{directed k-torus}). For a strongly connected directed graph $G$, we compute $\vv{\TC}_n(G)$ (see Proposition \ref{directed graphs}). We consider the Fadell--Neuwirth fibrations and the Hopf fibrations with a canonical $d$-structure and compute the associated sequential directed parametrized complexities (see Example \ref{directed Fadell--Neuwirth} and \ref{Hopf fibration}).

%In Section~\ref{Seq directed TC}, we introduce the sequential analogue of directed topological complexity (see~\cite{goubault2017directed,EGMFAS2020}). We investigate its basic properties, compare it with the usual sequential topological complexity, and compute the invariant for the directed loop and the directed circle equipped with particular directed structures. We also establish a relationship between sequential directed topological complexity and the directed Lusternik--Schnirelmann category introduced in~\cite{SDNDAS2025}. Furthermore, we introduce the notion of sequential dihomotopy equivalence and prove the corresponding invariance property.

%In Section~\ref{seq directed parametrized TC}, we introduce sequential directed parametrized topological complexity and study its fundamental properties in the fibrewise setting. We further define sequential fibrewise basic dihomotopy equivalence and show that sequential directed parametrized topological complexity is invariant under this equivalence relation.

\section{Preliminaries on Directed (parametrized) Topological Complexity}
In this section, we recall the basic notions from directed algebraic topology, specifically from directed (parametrized) topological complexity, that will be used throughout the paper. These include directed spaces, directed maps, directed homotopies, basic dihomotopy equivalences, directed Lusternik-Schnirelmann category, and directed (parametrized) topological complexity (see \cite{borat2020directed, SDNDAS2025,goubault2017directed,EGMFAS2020,MG2009,ioanpop}). In the later part, we define and study some of their sequential analogues.

\begin{definition}[\cite{MG2009}]\label{directed_space}
    A directed topological space (or $d$-space) is a pair $(X, dX)$ where $X$ is a topological space and $dX \subseteq X^I$ is a set of continuous paths from the unit interval $I$ to $X$, called directed paths (or $d$-paths), satisfying the following conditions:
    \begin{itemize}
        \item All constant paths belong to $dX$.
        \item $dX$ is closed under composition with any continuous, non-decreasing map $I \to I$.
        \item $dX$ is closed under path concatenation.
    \end{itemize}
\end{definition}
We denote a $d$-space $(X,dX)$ simply by $X$ when the context is clear. For any points $x,x'\in X$, let $dX(x,x')$ denote the subspace of $dX$ consisting of all directed paths $\gamma$ satisfying
\[
\gamma(0)=x \quad \text{and} \quad \gamma(1)=x'.
\]
We further define
\[
\Gamma_X=\{(x,x')\in X\times X \mid dX(x,x')\neq\varnothing\}\subseteq X\times X.
\]

The standard directed interval $\vv{I}= \vv{[0,1]}$ has the Euclidean
topology and the natural order. A (directed) path in a preordered
space $X$ is, by definition, a map $ \gamma\colon \vv{I} \to X$. 
\begin{definition}
    A continuous map $f\colon (X,dX)\to (Y,dY)$ is a $d$-map if $f\circ \gamma \in dY$ for any $\gamma\in dX$. We denote the corresponding induced map by $df\colon dX\to dY$. 
\end{definition}

\begin{definition}
    A directed homotopy between the $d$-maps $f\colon X\to Y$ and $g\colon X\to Y$ is a continous map $H\colon X\times \vv{I}\to Y$ such that $H(x,0)=f(x)$ and $H(x,1)=g(x)$ for all $x\in X$.
\end{definition}
This yields the notion of $d$-homotopy inverse and $d$-homotopy equivalence.

Now suppose that $(f,g)$ are $d$-homotopy inverses between $d$-spaces $X$ and $Y$. Then the restricted map
\[
dg_{f(x),f(x')}\colon dY(f(x),f(x'))\longrightarrow dX((g\circ f)(x),(g\circ f)(x'))
\]
does not, in general, yield an element of $dX(x,x')$. This motivates the introduction of a stronger notion of equivalence for $d$-spaces, called the \emph{basic dihomotopy equivalence}. 

A $d$-homotopy equivalence alone is generally insufficient to compare the directed path spaces associated with different pairs of points. To overcome this difficulty, one considers families of maps between directed path spaces that depend continuously on the endpoints, namely \emph{continuously graded maps}. Using this additional structure, Goubault \emph{et al.}~\cite{EGMFAS2020} introduced the notion of a \emph{basic dihomotopy equivalence}, which strengthens the classical notion of $d$-homotopy equivalence by requiring compatible homotopy equivalences between directed path spaces. %We now recall the definition of a continuously graded map. 
For the detailed descriptions and generalizations, see Section \ref{seq dihomotopy eqiv}, where we study its sequential version.

\begin{definition}
A $d$-map $p\colon E\to B$ has the $d$-homotopy lifting property with respect to a $d$-space $X$ if for every $d$-maps $f\colon X\to E$ and $\varphi\colon X\times\vv{I}\to B$ satisfying
\[
\varphi(x,\alpha)=p(f(x)),\qquad \alpha\in\{0,1\},
\]
there exists a $d$-map $\varphi':X\times\vv{I}\to E$ such that
$p\circ\varphi'=\varphi,$ and
$\varphi'(x,\alpha)=f(x)$.\\
A $d$-map $p$ is a $d$-fibration if it has the $d$-homotopy lifting property with respect to every $d$-space.
\end{definition}

The appropriate notion of equivalence for $d$-fibrations is the \emph{fibrewise basic dihomotopy equivalence}, introduced in \cite{SDNDAS2025} as a fibrewise extension of basic dihomotopy equivalence for $d$-spaces. It was shown there that directed parametrized topological complexity is invariant under this equivalence. In Section~\ref{fibrewise n-dihomotopy equivalence}, we introduce and study its sequential analogue, establishing the corresponding invariance results.

We now recall the directed topological invariants relevant to our work.
\begin{definition}[\cite{EGMFAS2020}]
The directed topological complexity $\vv{\TC}(X)$ of a $d$-space $X$ is the least natural number $n$ ($\infty$ if no such $n$ exists) such that $\Gamma_X$ admits a pairwise disjoint ENR cover $\{U_1,\ldots,U_n\}$ and the endpoint map
\[
\vv{\pi}\colon dX\to\Gamma_X,\qquad
\gamma\mapsto(\gamma(0),\gamma(1)),
\]
admits a continuous section over each $U_i$. 
\end{definition}

\begin{definition}[\cite{SDNDAS2025}]
Let $p\colon E\to B$ be a $d$-fibration. The directed parametrized topological complexity of $p$, denoted by $\vv{\TC}[p\colon E\to B]$, is the least natural number $n$ such that
\[
\Gamma_{E,B}
=\{(e_1,e_2)\in E\times_B E\mid \exists\,\gamma\in dE_B
\text{ with }\gamma(0)=e_1,\ \gamma(1)=e_2\}
\]
admits a cover by ENRs $\Gamma_{E,B}=U_1\cup\cdots\cup U_n$,
over each of which the directed parametrized endpoint map
\[
\vv{\Pi}\colon dE_B\to\Gamma_{E,B},\qquad
\gamma\mapsto(\gamma(0),\gamma(1)),
\]
admits a continuous section, where $dE_B=\{\gamma\in dE\mid p\circ\gamma \text{ is constant}\}$.
If no such $n$ exists, we set $\vv{\TC}[p\colon E\to B]=\infty$.
\end{definition}

The classical Lusternik--Schnirelmann category admits the following directed analogue, introduced in \cite{SDNDAS2025}.
\begin{definition}[\cite{SDNDAS2025}]
The directed Lusternik-Schnirelmann category of a $d$-space $X$, denoted by $\vv{\ct}(X)$, is the least natural number $n$ such that $X$ admits a cover by ENRs
\[
X=U_1\cup\cdots\cup U_n,
\]
and the evaluation map
\[
e_X\colon d_0X\to X,\qquad \gamma\mapsto\gamma(1),
\]
where $d_0X=\{\gamma\in dX\mid \gamma(0)=x_0\}$ for some fixed $x_0\in X$, admits a continuous section over each $U_i$. If no such $n$ exists, we set $\vv{\ct}(X)=\infty$.
\end{definition}

Unlike the classical (parametrized) topological complexity, the directed setting does not admit the standard inequalities involving the Lusternik--Schnirelmann category. Instead, directed (parametrized) topological complexity satisfies the following properties.
\begin{enumerate}
    \item $\vv{\ct}(X)\leq \vv{\TC}(X)$ and $\TC(X)\leq \vv{\ct}(X\times X)$, if $X$ has an initial point $x_0$ (that is, $dX(x_0,x)\neq\varnothing$ for every $x\in X$).
\item For a $d$-fibration $p\colon E\to B$, we have $\TC[p\colon E\to B]\leq \vv{\ct}(E\times_B E)$, provided $E\times_B E$ has an initial point. Moreover, if the fibre $F$ of the d-fibration is strongly connected, then we have
    $ \TC[p\colon E\to B]\leq \vv{\TC}[p\colon E\to B].$
\end{enumerate}

\section{Sequential directed topological complexity} \label{Seq directed TC}
In this section we consider a sequential analogue of the notion of directed topological complexity (see \cite{goubault2017directed}, \cite{EGMFAS2020}). We study its properties and compare it with the usual sequential topological complexity. We define a sequential analogue of dihomotopy equivalence (see  \cite[Section 7]{EGMFAS2020}) and show the sequential dihomotopy invariance property.

Let $(X, dX)$ be a $d$-space.	
Define the $n$-th $d$-paths map $\vv{\pi}_n\colon dX\to X^n$ by 
\[\vv{\pi}_n(\gamma)=\left(\gamma(0),\gamma(\frac{1}{n-1}),\cdots,\gamma(\frac{k}{n-1}), \cdots,\gamma(\frac{n-2}{n-1}),\gamma(1)\right).\]

The image of $\vv{\pi}_n$ is a subspace of $X^n$, denoted by
 $$\Gamma_{X}^n=\{(x_1,x_2,\dots, x_n)\colon \exists~\gamma\in dX \text{ such that } \gamma(\frac{k}{n-1})=x_{k+1} \text{ for } k=0,1,...,n-1 \}.$$

For $n=2$, the map $\vv{\pi}_n$ is the usual $d$-paths map defined in \cite[Page 14]{EGMFAS2020}. 
\subsection{Sequential directed topological complexity}
In the following, we define the sequential analogue of directed topological complexity and study its properties. 
\begin{definition}
The \(n\)-th directed  topological complexity of a \(d\)-space \((X, dX)\), denoted by \(\vv{\TC}_n(X)\), is the smallest natural number \(k\) (or \(\infty\) if no such \(k\) exists) such that the space \(\Gamma_{X}^n\) can be covered by \(k\) disjoint ENRs:
\[
\Gamma_{X}^n = F_1 \cup F_2 \cup \dots \cup F_k, \quad F_i \cap F_j = \emptyset \text{ for } i \ne j,
\]
and there exists a map $s\colon \Gamma_{X}^n \to dX$
satisfying the following conditions:
\begin{itemize}
    \item \(\vv{\pi}_n \circ s = Id\),
    \item for each \(i\), the restriction \(s_i= s|_{F_i} \colon F_i \to dX\) is continuous.
\end{itemize}
\end{definition}
\begin{example}
 \normalfont{Let $I=[0,1]$ be endowed with the $d$-structure given by the collection of all non-decreasing paths. We note that
\[
\Gamma_I^n=\{(x_1,\dots,x_n)\in I^n \mid x_i\leq x_{i+1}\text{ for all }1\leq i\leq n-1\}
\]
and the map $\vv{\pi}_n\colon dI\to I^n$ admits a continuous section
$s\colon \Gamma_I^n\to dI$, given by
\[
s(x_1,\dots,x_n)(t) = 
\begin{cases}
(x_2-x_1)(n-1)t + x_1 & 0 \leq t \leq \frac{1}{n-1}, \\
(x_3-x_2)(n-1)t + 2x_2 - x_3 & \frac{1}{n-1} \leq t \leq \frac{2}{n-1}, \\
\quad \vdots & \\
(x_{k+1}-x_k)(n-1)t + kx_k - (k-1)x_{k+1} & \frac{k-1}{n-1} \leq t \leq \frac{k}{n-1}, \\
\quad \vdots & \\
(x_n-x_{n-1})(n-1)t + (n-1)x_{n-1} - (n-2)x_n & \frac{n-2}{n-1} \leq t \leq 1,
\end{cases}
\]
where $1\leq k\leq n-1$.
%{\color{red}\[
%s(x_1,\dots,x_n)(t) =(x_{k+1}-x_k)(n-1)t+kx_k-(k-1)x_{k+1},
%\]}
%for $\frac{k-1}{n-1}\leq t\leq \frac{k}{n-1}$ and $1\leq k\leq n-1$. 
Hence $\vv{\TC}_n(I)=1$.}
\end{example}

By \cite[Proposition 3.3]{RUD2010}, we have $\TC_n(X)\leq \TC_{n+1}(X)$. The following proposition shows that the directed analogue also holds. Before this, we recall that a point $x_0\in X$ is said to be an \emph{initial} point if $(x_0,x)\in \Gamma_X^2$ for all $x\in X$.
\begin{proposition}\label{Properties_higher_TC}
Let $(X,dX)$ be a $d$-space with initial point $x_0$. Then
\[\vv{\TC}_n(X)\leq \vv{\TC}_{n+1}(X).\]
\end{proposition}
\begin{proof}
     Suppose $(X,dX)$ be a $d$-space with an initial point $x_0$ and $\vv{\TC}_{n+1}(X)=k$. Then we have a partition of $\Gamma_{X}^{n+1}=\bigcup_{i=1}^k U_i$ into ENRs such that there exists a map $s\colon \Gamma_{X}^{n+1}\to dX$ such that $\vv{\pi}_{n+1}\circ s=Id$ and $s\vert_{U_i}\colon U_i\to dX$ is continuous. Now define ENRs for $1\leq i\leq k$,
\[\widetilde{U_i}=\{(x_1,x_2,\dots,x_n)\in X^n\mid (x_0,x_1,\dots,x_n)\in U_i\}\] and \[s_i'\colon\widetilde{U_i}\to d\widetilde{U_i} \text{ by } s_i'(x_1,x_2,\dots,x_n)=s\vert_{U_i}(x_0,x_1,\dots,x_n).\] 
%We note that $\overbrace{(x,\dots,x)}^{n-\text{times}}\in \Gamma_X^n$. Since $x_0$ is an initial point, $(x_0,x,\dots,x)\in \Gamma_X^{n+1}$. Therefore $(x,\dots,x)\in \widetilde{U_i}$ for some $1\leq i\leq k$. 
Since $x_0$ is an initial point, Definition~\ref{directed_space} implies that if $(x_1,\dots,x_n)\in \Gamma_X^n$, then $(x_0,x_1,\dots,x_n)\in \Gamma_X^{n+1}$.
Hence the collection $\{\widetilde{U_i}\mid 1\leq i\leq k\}$ covers $\Gamma_{X}^{n}$. Moreover $s_i'$ is a continuous section of $\vv{\pi}_n\colon dX\to X^n$ on $\widetilde{U_i}$, for each $1\leq i\leq k$. Therefore, $\vv{\TC}_n(X)\leq \vv{\TC}_{n+1}(X)$.
\end{proof}
In \cite[Example 2]{EGMFAS2020}, the directed topological complexity of a particular directed circle was computed. In the following example, we study its sequential analogue and determine the corresponding sequential directed topological complexity.
\begin{proposition}\label{directed_circle}
Let $\vv{\mathbb{S}^1}$ denote the directed circle, where $\mathbb{S}^1$ is equipped with the directed structure consisting of all continuous paths $\gamma\colon [0,1]\to \mathbb{S}^1$ of the form
\[
\gamma(t)=e^{i\alpha(t)} \quad \text{or} \quad \gamma(t)=e^{-i\alpha(t)},
\]
where $\alpha\colon [0,1]\to [0,\pi]$ is a non-decreasing continuous function. Then $\vv{\TC}_{n}(\vv{\mathbb S^1})=n$ for all $n\geq 2$.
\end{proposition}
\begin{proof}
We prove this result by induction. The case $n=2$ follows from \cite[Example 2]{EGMFAS2020}; see also \cite{borat2020directed}.

Now, suppose that the result is true for all $k\leq n-1$. Then by Proposition~\ref{Properties_higher_TC}, we have $n-1\leq \vv{\TC}_{n}(\vv{\mathbb S^1})$.

We note that $\Gamma_{\vv{\mathbb{S}^1}}^n=\Gamma_{I_{+}}^n\cup \Gamma_{I_{-}}^n$ and $\Gamma_{I_{+}}^n\cap \Gamma_{I_{-}}^n=\{(b,\dots,b),(e,\dots,e)\}\cup B$, where 
\[B=\{(x_1,\dots,x_n)\mid x_1=\cdots =x_i=b,x_{i+1}=\cdots =x_n= e \text{ for } 1\leq i\leq n-1\}.\]

For
$u_k
=
(\underbrace{b,\dots,b}_{k},
\underbrace{e,\dots,e}_{n-k})$, where $1\le k\le n-1$,
the fiber $(\vv{\pi}_n)^{-1}(u_k)$ has exactly two connected components: one consisting of directed paths contained in $I_+^n$ and the other consisting of directed paths contained in $I_-^n$.

We claim that $\vv{\TC}_n(\vv{\mathbb S}^1)\geq n$. Assume, to obtain a contradiction, that $\vv{\TC}_n(\vv{\mathbb S}^1)\le n-1$.
Then there exist open sets
$U_1,\dots,U_{n-1}\subseteq \Gamma_{\vv{\mathbb S}^1}^n$
covering $\Gamma_{\vv{\mathbb S}^1}^n$ together with the continuous local sections of $\vv{\pi}_n$
$s_i\colon U_i\longrightarrow d\vv{\mathbb S}^1.$
The points $u_1,\dots,u_{n-1}$ are pairwise distinct elements of $\Gamma_{I_+}^n\cap \Gamma_{I_{-}}^n$. 
For $1\le k\le n-1$, choose an open neighbourhood $N_k$ of $u_k$ such that $N_k\cap N_l=\varnothing
\text{ for }
k\neq l$.

Since the sets $\{U_i\mid i=1,\dots ,n-1\}$ cover $\Gamma_{\vv{\mathbb S}^1}^n$, for every $k$ there exists some index $i(k)$ such that $F_{i(k)}\cap N_k\neq \varnothing.$
We claim that if
$i(k)=i(l)$
for distinct integers $k,l\in\{1,\dots,n-1\}$, then the section over $F_{i(k)}$ cannot be continuous.
Indeed, fix points
\[
\textbf{x}_k^+\in (\Gamma_{I_+}^n\setminus \Gamma_{I_{-}}^n)\cap N_k,
\qquad
\textbf{x}_k^-\in (\Gamma_{I_{-}}^n\setminus \Gamma_{I_+}^n)\cap N_k,
\]
and similarly choose
\[
\textbf{x}_l^+\in (\Gamma_{I_+}^n\setminus \Gamma_{I_{-}}^n)\cap N_l,
\qquad
\textbf{x}_l^-\in (\Gamma_{I_{-}}^n\setminus \Gamma_{I_+}^n)\cap N_l.
\]

Over points of $\Gamma_{I_{+}}^n\setminus \Gamma_{I_{-}}^n$, the fibers of $\vv{\pi}_n$ consist entirely of directed paths contained in $I_+^n$, while over points of $\Gamma_{I_{-}}^n\setminus \Gamma_{I_{+}}^n$, the fibers consist entirely of directed paths contained in $I_-^n$. %Hence continuity of a local section on $F_{i(k)}$ forces the section to choose consistently one of the two connected components of the fibers near $u_k$. The same holds near $u_\ell$.

Suppose that a single local section is defined on neighbourhoods of both $u_k$ and $u_l$. Since every neighbourhood of these points intersects both $\Gamma_{I_{+}}^n\setminus \Gamma_{I_{-}}^n$ and $\Gamma_{I_{-}}^n\setminus \Gamma_{I_{+}}^n$, continuity would force the section to extend simultaneously across the two connected components of the corresponding fibers of $\vv{\pi}_n$. This is impossible, since these components correspond to directed paths contained in $I_+^n$ and $I_-^n$, respectively.
%Since the fibers over $u_k$ and $u_l$ have disconnected components corresponding to incompatible branch choices, no single continuous local section can simultaneously extend across neighbourhoods of both $u_k$ and $u_\ell$. 
Therefore, $i(k)\neq i(l)$ when $k\neq \ell.$

Thus the distinct points $u_1,\dots,u_{n-1}$ require distinct local domains. Finally, at least one additional local domain is required to cover the regular part $(\Gamma_{I_+}^n\setminus \Gamma_{I_{-}}^n)\cup(\Gamma_{I_{-}}^n\setminus \Gamma_{I_+}^n)$. Hence 
$\vv{\TC}_n(\vv{\mathbb{S}^1})\geq n$.

%------------------------------------------------------------------------

%We note that $\Gamma_{\vv{\mathbb{S}^1}}^n=\Gamma_{I_{+}}^n\cup \Gamma_{I_{-}}^n$ and $\Gamma_{I_{+}}^n\cap \Gamma_{I_{-}}^n=\{(b,\dots,b),(e,\dots,e)\}\cup B$, where 
%\[B=\{(x_1,\dots,x_n)\mid x_1=\cdots =x_i=b,x_{i+1}=\cdots =x_n= e \text{ for } 1\leq i\leq n-1\}.\]

%Consider $(b,e,e,\dots,e)\in B$. Let $U\subset \Gamma_{\vv{\mathbb{S}^1}}^n$ be an open set containing $(b,e,\dots,e)$. Then $U$ must contain $(x_1^{+},\dots,x_n^{+})\in \Gamma_{I_+}^n$ and $(x_1^{-},\dots,x_n^{-})\in \Gamma_{I_-}^n$. Hence we may find sequences $(x_1^{\pm},\dots,x_n^{\pm})\in \Gamma_{I_{\pm}}^n$ converging to $(b,e,\dots,e)$.

For any $n \geq 2$, consider the following cover of $\Gamma_{\vv{\mathbb{S}^1}}^n$ into ENR's
$\Gamma_{\vv{\mathbb{S}^1}}^n=F_1\cup F_2\cup \cdots \cup F_n,$
where
$F_1=\Gamma^n_{I_+} - \{(b, \dots, b, e), (b, \dots, b, e, e), \dots, (b, e, \cdots, e)\},$
and for $2 \leq i \leq n$
$$F_i= \Gamma^n_{I_-} - \{(b, \dots, b), (\underbrace{b, \dots, b}_{n-i+1 \text{ times}}, \underbrace{e, \dots, e}_{i-1 \text{ times}}), (e, \dots,e)\}.$$  
Note that, the $d$-path map $\vv{\pi_n}$ has local continuous section on each $F_i$, for $1 \leq i \leq n$.
Therefore, $\vv{\TC}_n(\vv{\mathbb{S}^1}) \leq n$.
%Also, we have 
%\[n=\ct((\mathbb{S}^1)^n)\leq \vv{{\ct}}((\vv{\mathbb{S}^1})^n)\leq \vv{\TC}_n(\vv{\mathbb{S}^1}).\]
\end{proof}

\subsubsection{On $n$-regular $d$-spaces and strongly connected spaces}
We introduce the notion of an \emph{$n$-regular $d$-space}, which serves as the sequential analogue of the regular $d$-spaces studied in \cite[Section 4]{EGMFAS2020}. We first establish a relationship between $\TC_n(-)$ and $\vv{\TC}_n(-)$, and then strengthen this result under the additional assumption that the underlying $d$-space is strongly connected ( \cite[Definition 5]{EGMFAS2020}, see the later part). We also show that these inequalities need not hold for $d$-spaces that are not strongly connected; see Proposition~\ref{TC(X)}.

\begin{definition}\label{n-regular d-space}
   A $d$-space $(X,dX)$ is called \emph{$n$-regular} if there exists a partition 
   \[\Gamma_{X}^n=\bigcup_{i=1}^k U_i,~ k=\vv{\TC}_n(X)\]
   into ENRs such that the map $\vv{\pi}_n\colon dX\to X^n$ admits a continuous section over each $U_i$, and moreover, the sets $\bigcup\limits_{i=1}^r U_i$ are closed for $r=1,2,\dots,k$.
\end{definition}
\begin{remark}\label{n-regular_feature}
In an $n$-regular $d$-space with ENRs $\{U_i\}_{i=1}^k$, we have $\Bar{U_i} \cap U_j = \emptyset$, whenever  $i < j$.
\end{remark}
\begin{proposition}\label{comparision}
Let $(X,dX)$ be an $n$-regular $d$-space with initial point $x_0$. Then
    \[\TC_n(X)\leq n \vv{\TC}_2(X)-n+1.\]
\end{proposition}
\begin{proof}
    Suppose $\vv{\TC}_2(X)=k$. Then as $X$ is $n$-regular, $\Gamma_{X}=\Gamma_X^2$ can be covered by $k$ disjoint ENRs $U_1,U_2,\dots, U_k$ and there exists a map $s\colon \Gamma_X\to dX$ such that for each $i=1,2\dots,k$, the restricted map $s_i=s\vert_{U_i}$ is continuous section of $\vv{\pi_2}\colon dX\to X^2$. Now, consider
    \[\widetilde{U_i}:=\{x\in X\mid (x_0,x)\in U_i\}.\]
    Then note that the collection $\{\widetilde{U_i}\mid  1\leq i\leq k\}$ covers $X$ and each $\widetilde{U_i}$ admits a continuous section $s_i'\colon \widetilde{U_i}\to d_0X$ of the evaluation map $e_X$. Define
    \[V_{i_1,i_2,\dots,i_n}=\widetilde{U_{i_1}}\times \widetilde{U_{i_2}}\times \dots\times \widetilde{U_{i_n}} ~~\text{ and }~~~~ W_t=\bigcup\limits_{i_1+i_2+\cdots+i_n=t}V_{i_1,i_2,\dots,i_n} \text{ for } n\leq t\leq kn.\]
Also, define $\sigma_{i_1.i_2,\dots,i_n}\colon V_{i_1,i_2,\dots,i_n}\to X^I$ by 
\[\sigma_{i_1.i_2,\dots,i_n}(x_1,x_2,\dots,x_n)=s_{i_1 i_2}(x_1,x_2)* s_{i_2 i_3}(x_2,x_3)*\cdots* s_{i_{n-1}i_n}(x_{n-1},x_n),\] where
$s_{i_j i_{j+1}}(x_j, x_{j+1}):=(s_{i_j}(x_j))^{-1}* s_{i_{j+1}}(x_{j+1})$ for $1\leq j\leq n-1$.

Note that $\sigma_{i_1,i_2,\dots,i_n}$ defines a continuous local section of $e_n\colon PX\to X^n$. Hence, we have a continuous section of $e_n\colon PX\to X^n$ over each $W_k$ for $n\leq t\leq kn$. Moreover, the collection $\{W_t\mid  n\leq t \leq kn\}$ covers $X^n$. Therefore, $\TC_n(X)\leq nk-n+1$, completes the proof.
\end{proof}
Using Proposition~\ref{Properties_higher_TC}, we get the following result.
\begin{corollary}
For any $d$-space $(X,dX)$ with initial point $x_0$, 
 we have $\TC_n(X)\leq n \vv{\TC}_n(X)-n+1$.
\end{corollary}
Note that for $n=2$ the above corollary recovers \cite[Proposition 3.6(3)]{SDNDAS2025}.

If $(X,dX)$ is strongly connected, i.e., $ \Gamma_{X}=X \times X$ the above relation simplifies as follows.%then the preceding estimate takes the following form.
\begin{proposition}\label{Comparision_TC_DirectedTC}
For a strongly connected $d$-space $(X,dX)$, we have $\TC_n(X)\leq \vv{\TC}_n(X)$.
\end{proposition}
\begin{proof}
    We first show that $\Gamma_{X}^{n}=X^n$ for a strongly connected $d$-space $X$. Clearly, $\Gamma_{X}^{n}\subseteq X^n$. Suppose $(x_1,x_2,\dots,x_n)\in X^n$. Then since $X$ is strongly connected $d$-space, there exists directed path $\gamma_i$ starting from $x_i$ and ending at $x_{i+1}$ for each $i=1,2,\dots,(n-1)$. Now, consider the directed path $\gamma= \gamma_1* \gamma_2*\cdots*\gamma_{n-1}.$
 Note that $$\gamma(0)=x_1,~\gamma(\frac{1}{n-1})=x_2,\cdots,\gamma(\frac{k}{n-1})=x_{k+1},\cdots,\gamma(1)=x_n.$$
This implies that $(x_1,x_2,\dots,x_n)\in \Gamma_{X}^n$. Thus $\Gamma_{X}^{n}=X^n$. 

Now if $\vv{\TC}_n(X)=k$, then $X^n=\Gamma_{X}^n$ can be covered by $k$ disjoint ENRs $F_1,F_2,\dots, F_k$ and for each $i=1,2,\dots, k$ there exists continuous section $s_i\colon F_i(\subseteq X^n)\to dX (\subseteq PX)$ of $\vv{\pi}_n\colon dX\to X^n$. The same $s_i$'s can serve as a continuous section of the fibration $\pi_n\colon X^I\to X^n$ over $F_i$. Therefore, $\TC_n(X)\leq \vv{\TC}_n(X)$.
\end{proof}
Next, we show that $\mathbb{S}^1$  equipped with a $d$-structure such that $\vv{\TC}_n(X)=1 <\TC_n(X)=n$.

\begin{proposition}\label{TC(X)}
Let $X=\mathbb{S}^1$ be equipped with the $d$-structure consisting of continuous paths $\gamma\colon [0,1]\to \mathbb{S}^1$ satisfying the following conditions:
\begin{itemize}
    \item if $\gamma(0)=1$, then $\gamma$ is constant;
    \item if $\gamma(0)\neq 1$, then $\lvert \gamma(t)+1 \rvert$ is non-decreasing.
\end{itemize}
    For the $d$-space $X$, we obtain $\vv{\TC}_n(X)=1$.
\end{proposition}
\begin{proof}
    From the $d$-structure on $X$, we observe that 
    $\Gamma_X^n= A\cup B,$
    where $A=\{(e^{i\theta_1},\dots,e^{i\theta_n})\mid 0\leq \theta_1\leq \theta_2\leq \dots\leq \theta_n\leq \pi\}$ and $B=\{(e^{i\theta_1},\dots,e^{i\theta_n})\mid \pi\leq \theta_n\leq \theta_{n-1}\leq\dots\leq \theta_1\leq 2\pi\}$.
 Moreover note that $A\cap B= \{\textbf{1}, \textbf{-1}\}$.

    We define $s_A\colon A\to dX$ by $s(e^{i\theta_1},\dots,e^{i\theta_n}):= \gamma_{\theta_1,\dots,\theta_n}$, where $\gamma_{\theta_1,\dots,\theta_n}\colon [0,1]\to X$ is as follows:
\[
\gamma_{\theta_1,\dots,\theta_n}(t)=
\begin{cases}
\textbf{1} & \theta_1=0,\\[0.4em]
e^{\,i\left((1-s)\theta_j+s\theta_{j+1}\right)}
& \displaystyle \frac{j-1}{n-1}\leq t\leq \frac{j}{n-1},
\end{cases}
\]
where $s=(n-1)t-(j-1), \text{ for } 1\leq j\leq n-1$.
Then $\gamma_{\theta_1,\dots,\theta_n}\!\left(\tfrac{j}{n-1}\right)=e^{i\theta_{j+1}}$, for each $0\leq j\leq n-1$,
and $\gamma$ is a directed path in $X$. Hence $\vv{\pi}_n\circ s_A= Id_A$.

We define $s_B\colon B\to dX$ by $s(e^{i\theta_1},\dots,e^{i\theta_n}):= \widetilde{\gamma}_{\theta_1,\dots,\theta_n}$, where $\widetilde{\gamma}_{\theta_1,\dots,\theta_n}\colon [0,1]\to X$ is as follows:
\[
\widetilde{\gamma}_{\theta_1,\dots,\theta_n}(t)=
\begin{cases}
\textbf{1} & \theta_1=2\pi,\\[0.4em]
e^{\,i\left((1-s)\theta_j+s\theta_{j+1}\right)}
& \displaystyle \frac{j-1}{n-1}\leq t\leq \frac{j}{n-1},
\end{cases}
\]
where $s=(n-1)t-(j-1)$ and $1\leq j\leq n-1$.
Then
$\widetilde{\gamma}_{\theta_1,\dots,\theta_n}\!\left(\tfrac{j}{n-1}\right)=e^{i\theta_{j+1}}$ for $0\leq j\leq n-1$,
and $\gamma$ satisfies the given $d$-structure conditions. Thus $\vv{\pi}_n\circ s_B= Id_B$. 

Since $s_A(\textbf{x})=s_B(\textbf{x})$ for all $\textbf{x}\in A\cap B$, the map $s \colon \Gamma_X^n\to dX$, given by $s|_A=s_A, ~s|_B=s_B$, defines a global section of $\vv{\pi}_n$. Therefore $\vv{\TC}_n(X)=1$.
\end{proof}
\begin{remark}
Since $\TC_n(\mathbb{S}^1)=n$ \cite[\S~4]{RUD2010}, Proposition~\ref{TC(X)} implies that $\vv{\TC}_n(X)<\TC_n(X)$.
In view of Proposition~\ref{Comparision_TC_DirectedTC}, this shows that $X$ is not strongly connected.
\end{remark}
%\begin{remark}
%Since $\TC_n(\mathbb{S}^1)=n$ \cite[\S 4]{RUD2010}, Proposition~\ref{TC(X)} shows that the inequality of Proposition~\ref{Comparision_TC_DirectedTC} fails for $X$ i.e., $\vv{\TC}_n(X)<\TC_n(X)$. Hence, $X$ is not strongly connected.
%\end{remark}

\subsection{Comparison with the directed LS category} \label{compare cat and TC}
In this subsection, we describe the relationship between the directed LS category introduced in \cite{SDNDAS2025} and the sequential directed topological complexity.
\begin{proposition}\label{Properties_of_higher_TC}
Fix \(n\in \mathbb{N}\), and let \(X\) be a directed space. Assume that either there exists a point \(x_0\in X\) such that
$(x_0,x_1,\dots,x_{n-1})\in \Gamma_X^n$
for every \((x_1,\dots,x_{n-1})\in X^{n-1}\) (sequential analogue of initial point), or that \(X\) is strongly connected. Then
\[
\cat(X^{n-1})\leq \vv{\TC}_n(X),
\]
where \(\cat(X)\) denotes the directed LS-category of \(X\). %introduced in \cite[Definition 3.1]{SDNDAS2025}.
\end{proposition}
\begin{proof}
     Let $\vv{\TC}_n(X)=k$. Then $\Gamma_X^n$ can be partitioned into $k$ pairwise disjoint ENRs $U_1,U_2,\dots,U_k$, such that each $U_i$ admits a continuous section $s_i\colon U_i\to dX$ of the $d$-path space map $\vv{\pi}_n$. For $1\leq i\leq n$, define the ENRs
\[V_i:=\{(x_1,x_2,...,x_{n-1})\in X^{n-1}\mid (x_0,x_1,...,x_{n-1})\in U_i\}\]
and define $t_i\colon V_i\to d_0V_i$ by $t_i(x_1,x_2,...,x_{n-1}):= s_i(x_0,x_1,...,x_{n-1})$.

Under the assumption on $X$, we clearly note that $X^{n-1}=\bigcup\limits_{i=1}^k V_i$. Moreover, each $t_i$ is continuous on $V_i$, and $e_{X^{n-1}}\vert_{V_i}\circ t_i= Id_{V_i}$. This implies  $\cat(X^{n-1})\leq \vv{\TC}_n(X).$
\end{proof}

\begin{proposition}\label{Properties_of_higher_TC_1}
Let $(X,dX)$ be a $d$-space with initial point $x_0$. Then  $\TC_n(X)\leq \vv{\ct}(X^n)$.
%\begin{itemize}
 %   \item[(a)] $\cat(X^{n-1})\leq \vv{\TC}_n(X)$. 
  %  \item[(b)] $\TC_n(X)\leq \cat(X^n)$.
%\end{itemize}
\end{proposition}
\begin{proof}
Let $U\subset X^n$ be an ENR with a section $s\colon U\to d_0X^n$ of the map $e_{X^n}\colon d_0 X^n\to X^n$. 
Since $X^n$ has initial point $(x_0,\ldots,x_0)$, $d_0 X^n= (d_0X)^n$. 
Then define $s'\colon U\to PX$ by 
\[s'(x_1,\ldots,x_n):= \gamma_1^{-1}*\gamma_2*\gamma_2^{-1}*\cdots*\gamma_{n-1}*\gamma_{n-1}^{-1}*\gamma_n,\]
where $\gamma_i= pr_i\circ s(x_1,\ldots,x_n)$.
Note that 
\[s'(x_1,\cdots,x_n)(t)=\begin{cases}
    \gamma_{i+1}^{-1}(2(n-1)t-2i) & \text{if } t\in [\frac{i}{n-1}, \frac{2i+1}{2(n-1)}],\\
    \gamma_{i+2}(2(n-1)t-(2i+1)) & \text{if } t\in [\frac{2i+1}{2(n-1)},\frac{i+1}{n-1}],
\end{cases}\]
and so $s'(x_1,\ldots,x_n)(\frac{i}{n-1})=x
_{i+1}$ for each $i=0,\ldots,n-1$. 
Hence it defines a section of the free path space fibration $\pi_n\colon PX\to X^n$. 
Now applying the same argument for the ENRs covering $X^n$, the result follows.
\end{proof}
\begin{remark}
  Propositions~\ref{Properties_of_higher_TC},~\ref{Properties_of_higher_TC_1} recover \cite[Proposition 3.6 (1),(2)]{SDNDAS2025}.  
\end{remark}

%\subsection{$n$-regular $d$-space} In this section, we consider a sequential analogue of regular d-spaces and study sequential directed topological complexity over such spaces.
%\begin{definition}[$n$-Regular $d$-space]\label{n-regular d-space}
 %  A $d$-space $(X,dX)$ is called \emph{$n$-regular} if there exists a partition 
 %  \[\Gamma_{X}^n=\bigcup_{i=1}^k U_i,~ k=\vv{\TC}_n(X)\]
  % into ENRs such that the map $\vv{\pi}_n\colon dX\to X^n$ admits a continuous section over each $U_i$, and moreover, the sets $\bigcup\limits_{i=1}^r U_i$ are closed for $r=1,2,\dots,k$.
%\end{definition}
%\begin{remark}\label{n-regular_feature}
%In an $n$-regular $d$-space with ENRs $\{U_i\}_{i=1}^k$, we have $\Bar{U_i} \cap U_j = \emptyset$, whenever  $i < j$.
%\end{remark}
\subsection{On product of directed spaces} \label{product d-spaces}
We now recall the standard directed structure on products of $d$-spaces and establish the associated sequential topological complexity results.

Given $d$-spaces $(X,dX)$ and $(Y,dY)$, their cartesian product $X\times Y$ carries a natural $d$-space structure in which a path
\[
\gamma\colon [0,1]\to X\times Y,
\qquad
\gamma(t)=(\gamma_X(t),\gamma_Y(t)),
\]
is directed when $\gamma_X\in dX$ and $\gamma_Y\in dY$.
%The cartesian product of $d$-spaces $(X,dX)$ and $(Y,dY)$ has a natural $d$-space structure. Any path $\gamma\colon [0,1]\to X\times Y$ has the form $\gamma(t)=(\gamma_X(t),\gamma_Y(t))$ and we declare $\gamma$ to be directed if $\gamma_X\in dX$ and $\gamma_Y\in dY$. 
\begin{lemma}\label{Gamma_product}
   For $d$-spaces $(X,dX)$ and $(Y,dY)$, we get $\Gamma_{X\times Y}^n= \Gamma_X^n \times \Gamma_Y^n$.
\end{lemma}
\begin{proof}
    Let $\left((x_0,y_0),(x_1,y_1),...,(x_{n-1},y_{n-1})\right)\in \Gamma_{X\times Y}^n$. Then there exists $\gamma\in d(X\times Y)$ such that $\gamma(\frac{k}{n-1})=(x_k,y_k)$ for each $k=0,1,...,n-1$. 

    Now note that \( pr_1 \circ \gamma \in dX \) and \( pr_2 \circ \gamma \in dY \), where
$(pr_1 \circ \gamma)\left(\tfrac{k}{n-1}\right) = x_k$ and $(pr_2 \circ \gamma)\left(\tfrac{k}{n-1}\right) = y_k$ for each $k = 0, 1, \ldots, n-1$.
This implies that $\left((x_0,x_1,...,x_{n-1}), (y_0,y_1,...,y_{n-1})\right)\in \Gamma_X^n\times \Gamma_Y^n.$

Suppose $\left((x_0,x_1,...,x_{n-1}), (y_0,y_1,...,y_{n-1})\right)\in \Gamma_X^n\times \Gamma_Y^n$. Then there exist $\gamma_X\in dX$ and $\gamma_Y\in dY$ such that $\gamma_X(\tfrac{k}{n-1})=x_k,~\gamma_Y(\tfrac{k}{n-1})=y_k$ for each $k=0,1,...,n-1$. Then $\gamma(t)=(\gamma_X(t),\gamma_Y(t))$ is a directed path in $X\times Y$ and $\gamma(\frac{k}{n-1})=(x_k,y_k)$ for $k=0,1,...,n-1$. Therefore, $\left((x_0,y_0),(x_1,y_1),...,(x_{n-1},y_{n-1})\right)\in \Gamma_{X\times Y}^n$. This completes the proof.
\end{proof}
\begin{proposition}\label{TC_Product}
For $n$-regular $d$-spaces $(X_i,dX_i)$ with $i=1,\dots,m$, we obtain the inequality
\[\vv{\TC}_n(X_1\times \dots\times X_m)\leq \sum_{i=1}^m \left(\vv{\TC}_n(X_i)-1\right)+1.\]
 %  $$\vv{\TC}_n(X\times Y)\leq \left((\vv{\TC}_n(X)-1)+(\vv{\TC}_n(Y)-1)\right)+1 = \left(\vv{\TC}_n(X)+\vv{\TC}_n(Y)\right)-1.$$   
\end{proposition}
\begin{proof}
The general statement follows by repeated application of the case \(m=2\). Hence, it is enough to prove
$\vv{\TC}_n(X\times Y)\leq \vv{\TC}_n(X)+\vv{\TC}_n(Y)-1$.\\
  Suppose \( \vv{\TC}_n(X) = k+1 \) and \( \vv{\TC}_n(Y) = l+1 \). Then there exist partitions
\[
\Gamma_X^n = \bigcup_{i=0}^k F_i \quad \text{and} \quad \Gamma_Y^n = \bigcup_{i=0}^l G_i,
\]
where each \( F_i \) and \( G_i \) is an ENR, and for each \( r = 0,1, \ldots, k \) and \( s = 0,1, \ldots, l \), the unions \( \bigcup\limits_{i=0}^k F_i \) and \( \bigcup\limits_{i=0}^l G_i \) are closed. 
Moreover, there exist continuous sections $s_i$ over $F_i$ and $t_i$ over $G_i$ of $\vv{\pi}_n^X\colon dX\to X^n$ and $\vv{\pi}_n^Y\colon dY\to Y^n$, respectively. 

Since \( \Gamma_{X \times Y}^n = \Gamma_X^n \times \Gamma_Y^n \) by Lemma~\ref{Gamma_product}, the sets
$\{ F_i \times G_j \mid 0 \leq i \leq k,\ 0 \leq j \leq l \}$
form an ENR partition of \( \Gamma_{X \times Y}^n \). 
Moreover, the continuous sections $s_i\colon F_i\to dX$ and $t_j\colon G_j\to dY$ produce continuous section $\sigma_{ij}=s_i\times t_j\colon F_i\times G_j\to d(X\times Y)$ of $\vv{\pi}_n^{X\times Y}\colon d(X\times Y)\to (X\times Y)^n$ over $F_i\times G_j$.

 Consider the sets 
  \begin{equation}\label{open_set}
        \bigcup_{i+j=t} (F_i\times G_j)= W_t\subseteq \Gamma_{X\times Y}^n 
  \end{equation}
with $t=0,1,...,k+l$. 
We note that the sets appearing in the union \eqref{open_set} are open in \(W_t\) and pairwise disjoint, by Remark \ref{n-regular_feature}. Consequently, the family of continuous maps \(\sigma_{ij}\) determines a continuous section \(W_t \to d(X\times Y)\). Therefore, 
  $\vv{\TC}_n(X\times Y)\leq (k+l)+1.$ This completes the proof.
%-----------------------------------------------------------------------
%   Suppose \( \vv{\TC}_n(X) = k \) and \( \vv{\TC}_n(Y) = l \). Then there exist partitions
%\[
%\Gamma_X^n = \bigcup_{i=1}^k F_i \quad \text{and} \quad \Gamma_Y^n = \bigcup_{i=1}^l G_i,
%\]
%where each \( F_i \) and \( G_i \) is an ENR, and for each \( r = 1, 2, \ldots, k \) and \( s = 1, 2, \ldots, l \), the unions \( \bigcup_{i=1}^r F_i \) and \( \bigcup_{i=1}^s G_i \) are closed. Moreover, there exist continuous sections $s_i$ over $F_i$ and $t_i$ over $G_i$ of $\vv{\pi}_n^X: dX\to X^n$ and $\vv{\pi}_n^Y: dY\to Y^n$, respectively. 
%Since \( \Gamma_{X \times Y}^n = \Gamma_X^n \times \Gamma_Y^n \), the sets
%\[
%\{ F_i \times G_j \mid 1 \leq i \leq k,\ 1 \leq j \leq l \}
%\]
%form an ENR partition of \( \Gamma_{X \times Y}^n \). Moreover, the continuous sections $s_i: F_i\to dX$ and $t_j: G_j\to dY$ produce continuous section 
%  $$\sigma_{ij}=s_i\times t_j: F_i\times G_j\to d(X\times Y)$$ of $\vv{\pi}_n^{X\times Y}: d(X\times Y)\to (X\times Y)^n$ over $F_i\times G_j$.
%  Consider the sets 
 % \begin{equation}\label{open_set}
  %      \bigcup_{i+j=t} (F_i\times G_j)= W_t\subseteq \Gamma_{X\times Y}^n 
 % \end{equation}
%with $t=2,...,k+l$. We note that the sets appearing in the union \eqref{open_set} are open in \(W_t\) and pairwise disjoint, by Remark \ref{n-regular_feature}. Consequently, the family of continuous maps \(\sigma_{ij}\) determines a continuous section \(W_t \to d(X\times Y)\). Therefore, 
  %$$\vv{\TC}_n(X\times Y)\leq (k+l)-1.$$ This completes the proof.
\end{proof}
\begin{corollary} \label{directed torus}
Let $\vv{\mathbb{S}^1}$ denote the directed circle considered in Proposition \ref{directed_circle}. Then, for the associated directed torus $(\vv{\mathbb{S}^1})^k$, we have
$\vv{\TC}_n\big((\vv{\mathbb{S}^1})^k\big)\leq k(n-1)+1$.
\end{corollary}
%\begin{corollary} \label{directed torus}
%In Proposition \ref{directed_circle}, we consider  the directed circle $\vv{\mathbb{S}^1}$ and computed $\vv{\TC}_n(\vv{\mathbb{S}^1})$. Then the associated directed torus $(\vv{\mathbb{S}^1})^k$ satisfies $\vv{\TC}_n((\vv{\mathbb{S}^1})^k)\leq k(n-1)+1$.
%\end{corollary}
%Next we consider the directed loop space $\mathbb{O}^1$ which is defined as the unit circle $\mathbb{S}^1$ with $d$-structure as follows. Any continuous path $\gamma\colon [0,1]\to \mathbb{S}^1$ can be written as $\gamma(t)=\mathrm{exp}(\iota\phi(t))$, where $\phi\colon [0,1]\to \mathbb{R}$ is defined uniquely up to adding an integer multiple of $\pm 2\pi$. We call the path $\gamma$ is positive if the function $\phi(t)$ is non-decresing.

We next consider the directed loop $\mathbb{O}^1$, defined as the unit circle $\mathbb{S}^1$ equipped with the following $d$-structure. Let $\gamma\colon [0,1]\to \mathbb{S}^1$ be a continuous path. Then $\gamma$ may be expressed as
\[
\gamma(t)=\exp(i\phi(t)),
\]
where $\phi\colon [0,1]\to \mathbb{R}$ is determined up to addition by an integral multiple of $\pm 2\pi$. We say that $\gamma$ is directed if the function $\phi$ is non-decreasing. We now compute the sequential directed topological complexity of the directed circle $\mathbb{O}^1$.
\begin{proposition}\label{TC_O^1}
    %Let $\mathbb{O}^1$ be the directed loop. Then 
    For the directed loop $\mathbb{O}^1$, we have
   $\vv{\TC}_n(\mathbb{O}^1)=n$.
\end{proposition}
\begin{proof}
    We note from Proposition \ref{Comparision_TC_DirectedTC} and \cite[Page 918]{RUD2010} that $\vv{\TC}_n(\mathbb{O}^1)\geq \TC_n(\mathbb{S}^1)=n.$
The $d$-space $\mathbb{O}^1$ being strongly connected \cite[Example 4]{EGMFAS2020}, $\Gamma_{\mathbb{O}^1}^n= \overbrace{\mathbb{O}^1\times \dots\times \mathbb{O}^1}^{n \text{ times}}$. %Also, we can partition %$$\overbrace{\mathbb{O}^1\times \dots\times \mathbb{O}^1}^{n \text{ times}}= F\cup \bigcup_{i=1}^{n-1} F_i,$$ where
   % \[F=\{(z_1,z_2,\dots,z_n): z_1=z_2=\dots = z_n\},\]
   % and 
   % \[F_i=\{(z_1,\dots,z_i,z_{i+1},\dots,z_n): z_i\neq z_{i+1}\}.\]

%$$\overbrace{\mathbb{O}^1\times \dots\times \mathbb{O}^1}^{n \text{ times}}= \bigcup_{i=1}^{n-1} F_i\cup F,$$ 
For $1\leq k\leq n-1$, we consider
\[
F_k=\{(z_1,\dots,z_n)\in (\mathbb{O}^1)^n \mid
z_k\neq z_{k+1},\ 
z_{k+1}=z_{k+2}=\cdots = z_n\},
\]
and 
\[
F=\{(z_1,\dots,z_n)\in (\mathbb{O}^1)^n \mid
z_1=z_2=\cdots=z_n\}.
\]

We show that $(\mathbb{O}^1)^n = F \cup F_1 \cup \cdots \cup F_{n-1}.$

Suppose $(z_1,\dots,z_n)\in (\mathbb{O}^1)^n$. If $z_1=z_2=\cdots=z_n$, then $(z_1,\dots,z_n)\in F$. Otherwise, let $k$ be the largest index such that $z_k\neq z_{k+1}$.
By maximality of $k$, $z_{k+1}=z_{k+2}=\cdots=z_n$. 
Hence $(z_1,\dots,z_n)\in F_k$. This proves our claim.

Now note that we can obtain a section of $\vv{\pi}_n \colon d\mathbb{O}^1 \longrightarrow  (\mathbb{O}^1)^n$ over \(F\) by assigning, to each point \((z,z,\dots,z)\in F\), the constant path at \(z\).

For each $(z_1,\dots,z_n)\in F_k$, choose angles $z_j=e^{i\theta_j}$, such that $\theta_1\le \theta_2\le \cdots \le \theta_n$, with $\theta_k<\theta_{k+1}$.

Define $s_k: F_k\longrightarrow d\mathbb{O}^1$ by
$s_k(z_1,\dots,z_n)(t)
=
\exp(i\Phi(t))$,
where \(\Phi\colon [0,1]\to \mathbb{R}\) is the piecewise linear map satisfying
$\Phi\!\left(\tfrac{j-1}{n-1}\right)=\theta_j, ~j=1,\dots,n.$

Since $\theta_1\le \theta_2\le \cdots \le \theta_n$, the function \(\Phi\) is non-decreasing. Hence $s_k(z_1,\dots,z_n)$
is a directed path in \(\mathbb{O}^1\). Moreover,
$\vv{\pi}_n\bigl(s_k(z_1,\dots,z_n)\bigr)
=
(z_1,\dots,z_n).$
Therefore \(s_k\) is a continuous section of \(\vv{\pi}_n\) over \(F_k\).
Hence, we obtain  
$\vv{\TC}_n(\mathbb{O}^1) \leq n$.
\end{proof}
%\begin{corollary} For the directed $k$-dimensional torus $(\mathbb{O}^1)^k$, we obtain
%    $\vv{TC}_n\left((\mathbb{O}^1)^k\right)=k(n-1)+1$.
%\end{corollary}
\begin{corollary} \label{directed k-torus}
The directed sequential topological complexity of the directed \(k\)-dimensional torus \((\mathbb{O}^1)^k\) is given by
$\vv{\TC}_n\bigl((\mathbb{O}^1)^k\bigr)=k(n-1)+1$.
\end{corollary}
\begin{proof}
 Since $\mathbb{O}^1$ is regular $d$-space \cite[Page 18]{EGMFAS2020}, it follows from Propositions \ref{TC_Product} and \ref{TC_O^1} that $\vv{\TC}_n\left((\mathbb{O}^1)^k\right)\leq k \vv{\TC}_n(\mathbb{O}^1)+1= k(n-1)+1.$

 $\mathbb{O}^1$ being strongly $d$-connected, Proposition \ref{Comparision_TC_DirectedTC} and \cite[Corollary 3.13]{IBJG_2014} implies that $\vv{\TC}_n\left((\mathbb{O}^1)^k\right)\geq \TC_n\left((\mathbb{S}^1)^k\right)=k(n-1)+1.$
\end{proof}

%Using the same idea as \cite[Proposition 3]{EGMFAS2020}, we obtain
%\begin{proposition}
 %   Let $G$ be a directed connected graph. Then $\vv{\TC}_n(G)\leq (n+1)$.
%\end{proposition}
\begin{proposition}\label{Upper_bound_dir_tc_graph}
    Let $G$ be a directed connected graph. Then $\vv{\TC}_n(G)\leq n+1$. 
\end{proposition}
\begin{proof}
    The proof follows using the same argument as \cite[Proposition 3]{EGMFAS2020}.
\end{proof}
\begin{proposition} \label{directed graphs}
    Let $G$ be a strongly connected directed graph. Then 
   \[\vv{\TC}_n(G)=\begin{cases}
      1 & \text{ if } b_1(G)=0,\\
       n & \text{ if } b_1(G)=1,\\
       n+1 & \text{ if } b_1(G)\geq 2,
   \end{cases}\]
    %    \[\vv{\TC}_n(G)\geq \text{min}\{b_1(G),n\}+1,\]
%    \[\vv{\TC}_n(G)\leq \text{min}\{b_1(G),n\}+1, \text{ if } b_1(G)=0,1 \text{ or } \geq n.\]
    where $b_1(G)$ denotes the first Betti nuber of the graph $G$.
\end{proposition}
\begin{proof}
    It is known that \[\TC_n(G)=\begin{cases}
         1 & \text{ if } b_1(G)=0,\\
       n & \text{ if } b_1(G)=1,\\
       n+1 & \text{ if } b_1(G)\geq 2.
   \end{cases}\]
    Since \( G \) is strongly connected, Proposition~\ref{Comparision_TC_DirectedTC} yields  
$\vv{\TC}_n(G) \ge \TC_n(G)$.  
We now show that equality holds by considering cases according to the value of \( b_1(G) \).
\begin{itemize}
    \item[(a)] Suppose $b_1(G)=0$. Then $G$ is contractible and strongly connected. Hence $\vv{\TC}_n(G)=1$.
    \item[(b)] Suppose \(b_1(G)=1\). Then \(G\) is a cycle. Consider the partition
$\Gamma_G^n=G^n=F_1\cup F_2\cup\cdots\cup F_n$,
where
\[F_i:=\Bigl\{
(x_1,\dots,x_n)\in G^n \;\Big|\;
\begin{array}{l}
\text{exactly } i \text{ coordinates are vertices of } G,\\
\text{and the remaining coordinates lie in edge interiors}
\end{array}
\Bigr\}.\]
Using the same argument as in \cite[Proposition 3]{EGMFAS2020}, one obtains a section of \(\vv{\pi}_n\) over each \(F_i\). Therefore, $\vv{\TC}_n(G)\leq n$.
Since \(\vv{\TC}_n(G)\geq \TC_n(G)=n\), it follows that
$\vv{\TC}_n(G)=n.$
    \item[(c)] Suppose $b_1(G)\geq 2$. Then $\vv{\TC}_n(G)\geq n+1$. On the other hand, Proposition~\ref{Upper_bound_dir_tc_graph} gives $\vv{\TC}_n(G)\leq n+1$. Therefore, $\vv{\TC}_n(G)=n+1$.

\end{itemize}
\end{proof}

\subsection{Sequential Dihomotopy Equivalence} \label{seq dihomotopy eqiv}
In this subsection, we introduce the sequential analogue of the notion of dihomotopy equivalences for directed spaces and show that sequential directed topological complexity is a sequential dihomotopy invariance property.

We know that for a contractible space \(X\), \(\TC_n(X)=1\). However, in the directed setting, the choice of \(d\)-structure may create nontrivial motion-planning obstructions. The next example shows that directed topological complexity is not a homotopy invariant and depends on the directed structure of the space. In particular, for a suitable d-structure on the disk, the directed higher topological complexity is not equal to $1$.
\begin{example}\label{directed_TC_D^2}
  \normalfont{ Let $X$ be $2$-disc $\mathbb{D}^2$ with the following $d$-structure:
    \begin{itemize}
        \item Any directed path starting at an interior point of $\mathbb{D}^2$ is constant.
        \item The directed path $\gamma\colon [0,1]\to \mathbb{D}^2$ satisfying $\lvert\gamma(0)\rvert=1$ can be either $\gamma(t)=e^{i\alpha(t)}$ or $\gamma(t)=e^{-i\alpha(t)}$, where $\alpha\colon [0,1]\to [0,\pi]$ is a non-decreasing continuous function.
    \end{itemize}
    Therefore, the boundary of $X$ is isomorphic to the directed circle $\vv{\mathbb{S}^1}$ as a $d$-space. With this \(d\)-structure, we have
$\Gamma_X^n = F \sqcup \Gamma_{\vv{\mathbb{S}}^1}^n$,
where
$F=\{(z,\dots,z)\mid z \text{ is an interior point of } \mathbb{D}^2\}$. 
Consequently, any section $s\colon G\to d\vv{\mathbb{S}}^1$ of $\vv{\pi}_n\colon d\vv{\mathbb{S}}^1\to (\vv{\mathbb{S}}^1)^n$, where \(G\subset \Gamma_{\vv{\mathbb{S}}^1}^n\), also defines a section of $\vv{\pi}_n\colon dX\to X^n$. Therefore it follows from Example~\ref{directed_circle} that $\vv{\TC}_n(X)\geq n$. Since $X$ is contractible, we have $\TC_n(X)=1$.}
\end{example}
The above example motivates us to define sequential dihomotopy equivalence.

For $(x_1,x_2,\dots,x_n)\in X^n$, we use 
$dX(x_1,x_2,\dots,x_n)$ to denote the subspace of $dX$ containing all $d$-paths $\gamma$ such that $$\gamma(0)=x_1,~\gamma(\frac{1}{n-1})=x_2,\ldots,\gamma(\frac{k}{n-1})=x_{k+1},\ldots,\gamma(1)=x_n.$$
We now define the notion of a continuously graded map in a sequential setting.
\begin{definition}
Let $f\colon X\to Y$ be a $d$-map and let $y_1,y_2,\dots,y_n \in Y$ and $W\subset X^n$ be the inverse image of $(y_1,y_2,\dots,y_n)$ under the map $\prod\limits_{i=1}^n f$. Suppose we have continuous maps 
\[F^{y_1,y_2,\dots,y_n}\colon dV(y_1,y_2,\dots,y_n)\times W\to dX\]
such that for all $(x_1,x_2,\dots,x_n)\in W$, we have 
$F^{y_1,y_2,\dots,y_n}(\gamma,x_1,x_2,\dots,x_n)\in dX(x_1,x_2,\dots,x_n)$.
In this case, we regard the map $F=(F^{y_1,y_2,\dots,y_n})$ to be continuously $n$-graded. We also denote the grading of this map by \[F_{x_1,x_2,\dots,x_n}\colon dY(f(x_1),f(x_2),\dots,f(x_n))\to dX(x_1,x_2,\dots,x_n)\] varying continuously over $(x_1,x_2,\dots,x_n)\in W$ in $dX^{dY(y_1,y_2,\dots,y_n)}$, with respect to the compact-open topology.    
\end{definition}
\begin{definition}%[$n$-basic dihomotopy equivalence] 
\label{n-basic dihomotopy equivalence}
    Let $(X,dX)$ and $(Y,dY)$ be two $d$-spaces. A $d$-map $f\colon X\to Y$ is said to be \emph{$n$-basic dihomotopy equivalence} if the following conditions are satisfied:
    \begin{itemize}
        \item $(f,g)$ is a $d$-homotopy equivalence between $X$ and $Y$.
        \item There exists a map $F\colon dY\to  dX$, continuously graded by \[F_{x_1,x_2,\dots,x_n}\colon dY(f(x_1),f(x_2),\dots,f(x_n))\to dX(x_1,x_2,\dots,x_n),\] for  $(x_1,x_2,\dots,x_n)\in \Gamma_{X}^{n}$, such that $(df_{x_1,x_2,\dots,x_n}, F_{x_1,x_2,\dots,x_n})$ is a homotopy equivalence between $dX(x_1,x_2,\dots,x_n)$ and $dY(f(x_1),f(x_2),\dots,f(x_n))$ .
        \item There exists a map $G\colon dX\to  dY$ continuously graded by \[G_{y_1,y_2,\dots,y_n}\colon dX(g(y_1),g(y_2),\dots,g(y_n))\to dY(y_1,y_2,\dots,y_n),\] for $(y_1,y_2,\dots,y_n)\in \Gamma_{Y}^{n}$, such that $(dg_{y_1,y_2,\dots,y_n}, G_{y_1,y_2,\dots,y_n})$ is a homotopy equivalence between $dY(y_1,y_2,\dots,y_n)$ and $dX(g(y_1),g(y_2),\dots,g(y_n))$.
    \end{itemize}
\end{definition}
As expected, the $n$-directed sequential topological complexity is invariant under $n$-basic dihomotopy equivalence, as stated in the following proposition.
\begin{proposition} \label{n-dihmotopy invariance}
    Let $X$ and $Y$ be two simply $n$-basic dihomotopy equivalent $d$-spaces. Then $$\vv{\TC}_n(X)=\vv{\TC}_n(Y).$$
\end{proposition}
\begin{proof}
    Since \( X \) and \( Y \) are $n$-basic dihomotopy equivalent, there exist \( d \)-maps \( f \colon X \to Y \) and \( g \colon Y \to X \) that form a \( d \)-homotopy equivalence between \( X \) and \( Y \). Moreover, we get a continuously graded map \( G \colon dX \to dY \) such that for all \( (y_1,y_2,\dots,y_n ) \in \Gamma_{Y}^{n} \), the map 
    $G_{y_1,y_2,\dots,y_n} \colon dX(g(y_1), g(y_2),\dots,g(y_n)) \to dY(y_1, y_2,\dots,y_n)$ 
is a homotopy inverse of  
$dg_{y_1, y_2,\dots,y_n} \colon dY(y_1, y_2,\dots,y_n) \to dX(g(y_1), g(y_2),\dots,g(y_n))$.

Suppose \( \vv{\TC}_n(X) = k \). Then there exists a partition  
$\Gamma_{X}^{n} = \bigcup_{i=1}^k F_i,$
where each \( F_i \) is an ENR and \( F_i \cap F_j = \emptyset \) for \( i \neq j \). Moreover, there exists a map  
$s \colon \Gamma_{X^{n-1}} \to dX$
such that \( \vv{\pi}_n^X \circ s = \mathrm{Id} \), and the restriction \( s_i := s|_{F_i} \) is continuous for each \( i \).
Define 
$$A_i=\{(y_1,y_2,\dots,y_n)\in \Gamma_{Y}^{n}\mid (g(y_1),g(y_2),\dots,g(y_n)) \in F_i\}$$
Clearly, $A_i$ is either empty or an ENR, and $\Gamma_{Y}^{n}=\bigcup_{i=1}^k A_i.$
For all $(y_1,y_2,\dots,y_n)\in A_i\subseteq \Gamma_{Y}^{n}$, define
$$t_i(y_1,y_2,\dots,y_n)= G_{y_1,y_2,\dots,y_n}\circ s_i(g(y_1),g(y_2),\dots,g(y_n)).$$ Then \( t_i \) is continuous on \( A_i \), since \( s_i \) is continuous on \( F_i \), \( g \) is continuous on \( Y \), and \( G \) varies continuously graded. Moreover, \( \vv{\pi}_n^Y \circ t_i = \mathrm{Id}_{A_i} \). Thus, $\vv{\TC}_n(Y)\leq \vv{\TC}_n(X).$ Reversing the roles of $X$ and $Y$ similarly gives $\vv{\TC}_n(X)\leq \vv{\TC}_n(Y)$. Therefore, $\vv{\TC}_n(X)=\vv{\TC}_n(Y).$
\end{proof}
In \cite[Theorem 1]{EGMFAS2020}, it is shown that the directed topological complexity of a dicontractible space is $1$, and the converse implication also holds. We now show that this result extends to sequential directed topological complexity.
\begin{proposition} \label{Important_Result}
    Let $X$ be a contractible $d$-space. Then the map $\vv{\pi}_n\colon dX\to X^n$ has a continuous global section if and only if $X$ is $n$-basic dihomotopy equivalent to a point.
\end{proposition}
\begin{proof}
Since $X$ is contractible, there exist maps $f\colon X\to \{a_0\}$ and $g\colon\{a_0\}\to X$, where $f$ is the constant map and $g$ is the inclusion, such that $(f,g)$ determines a classical homotopy equivalence. Clearly, both $f$ and $g$ are $d$-maps, and hence $(f,g)$ is a $d$-homotopy equivalence.
 % As $X$ is contractible, we have $f:X\to \{a_0\}$ (the constant map) and $g: \{a_0\}\to X$ (the inclusion) which form a (classical) homotopy equivalence. Trivially, $f$ and $g$ are $d$-maps and $(f,g)$ is a $d$-homotopy equivalence. 

  Suppose $s$ is a continuous section of $\vv{\pi}_n$. There is an obvious inclusion map,which is continuously $n$-graded in $x_1,x_2,\dots,x_n$, given by $i\colon \{s(x_1, x_2,\dots,x_n)\}\to dX(x_1,x_2,\dots,x_n)$. Define $R$ to be that map. Now the constant map $r\colon dX(x_1,x_2,\dots,x_n)\to \{s(x_1,x_2,\dots,x_n)\}$ is a retraction map for $i$. Consider the homotopy $ H\colon dX \times [0,1]\to dX$ defined by $H(u,t):=v^t$, where
%  \begin{equation*} 
 %   \begin{aligned}
  %  H\colon dX \times [0,1] &\rightarrow dX \quad
   % \hspace{1em} \left(u, t\right) &\mapsto v^t,
   % \end{aligned}
%\end{equation*}
%defined by
\[v^t(x)=\begin{cases}
    u\left(\frac{(t+1)(t+2)...(t+n-2)}{(n-1)!} x\right) & \mbox{if $0\leq x\leq t$},\\
   s(u_0(t),u_1(t),\dots,u_{n-2}(t), u(1))(\frac{x-t}{1-t})&\mbox{if $t\leq x\leq 1$.}
\end{cases}\]
Here $u_i(t):= u\left(\tfrac{(t+i)(t+i+1)...(t+n-2)}{(n-1)!}i!\right)$, for $i=0, \dots, n-2.$

Notice that, since the reparameterization of the interval $[0, 1]$ by $t \mapsto \tfrac{(t+i)(t+i+1)...(t+n-2)}{(n-1)!}i!$ is non-decreasing, $u_i \in dX$. %As, concatenation and evaluation are continuous and as $s$ is continuous in each of the arguments $H$ is continuous in $U \in dX$ and in $t$. 
Moreover, since concatenation and evaluation are continuous, and $s$ is continuous in each variable, it follows that $H$ is continuous in both $U\in dX$ and $t$. 
$H$ induces families $H_{x_1, \dots, x_n} \colon d X(x_1, \dots, x_n) \times [0, 1] \rightarrow d X(x_1, \dots, x_n),$ and because $H$ is continuous in $u$ in the compact-open topology, this family $H_{x_1, \dots, x_n}$ is continuous in $x_1, \dots, x_n$  in $X$.

We also have \( H(u,1) = u \) and \( H(u,0) = s(u_0(0), u_1(0), \dots, u_{n-2}(0), u(1)) = i \circ r(u) \). Therefore, \( r \) is a deformation retraction, and \( dX(x_1, x_2, \dots, x_n) \) is homotopy equivalent to the point \( \{s(x_1, x_2, \dots, x_n)\} \). In particular, it is contractible for all \( (x_1, x_2, \dots, x_n) \in \Gamma_X^n \), so \( R \) is a (graded) homotopy equivalence.

Conversely, suppose $X$ is $n$-basic dihomotopy equivalent to a point. Then there exists a continuous map $R\colon \{*\}\to dX$, which is graded in $(x_1,x_2,\dots,x_n)\in \Gamma_X^n$. Define $$s(x_1,x_2,\dots,x_n):=R_{x_1,x_2,\dots,x_n}(*)~~\text{~~ for } (x_1,x_2,\dots,x_n)\in \Gamma_X^n.$$ Then this is a continuous section of the $n$-th $d$-paths map $\vv{\pi}_n\colon dX\to X^n$. This completes the proof.
\end{proof}

\section{Sequential directed parametrized topological complexity} \label{seq directed parametrized TC}

We define a sequential analogue of directed parametrized topological complexity  (see \cite{SDNDAS2025}), call it as the sequential directed parametrized topological complexity. We extend some of the important results associated with directed fibrations in the sequential context.

Given a $d$-fibration $p\colon E\to B$ with $E$ and $B$ are ENRs, we define

\[dE_B:=\{\gamma\in dE\mid p\circ\gamma \text{ is constant}\},\]
\[E_B^n=\{(e_1,e_2,\dots,e_{n})\in \prod_{i=1}^n E \mid p(e_1)=p(e_2)=\dots=p(e_{n})\}\]
and 
\[
\Gamma_{E,B}^n := \left\{ (e_1,e_2, \dots, e_{n}) \in E_B^n\;\middle|\;
\begin{aligned}
&\exists\, \gamma \in dE_B \text{ such that } \gamma\left(\tfrac{k}{n-1}\right) = e_{k+1} \\
&\text{for } k = 0, \dots, n-1\; 
\end{aligned}
\right\}.
\]

We consider the following map
\begin{equation*} 
    \begin{aligned}
    \vv{\Pi}_n\colon dE_B &\rightarrow \Gamma_{E,B}^n \subseteq E_B^n \quad
    \hspace{1em} \gamma &\mapsto \left(\gamma(0),\gamma(\frac{1}{n-1}),\cdots,\gamma(\frac{k}{n-1}),\cdots,\gamma(1)\right).
    \end{aligned}
\end{equation*}

\subsection{Sequential directed parametrized topological complexity} We define the sequential analogue of directed parametrized topological complexity, in order to study sequential motion planning algorithms in the parametrized settings of directed spaces.
\begin{definition}
   The \emph{$n$-directed parametrized topological complexity} (DPTC) of a $d$-fibration $p\colon E\to B$, denoted by $\vv{\TC}_n[p\colon E\to B]$, is defined as the smallest natural number $k$ (infinity if it does not exist) such that $\Gamma_{E,B}^n$ is covered by $k$ ENRs $F_1,F_2,\dots,F_k$ with each ENR $F_i$ admits continuous section of the map $\vv{\Pi}_n$ and $F_i\cap F_j=\emptyset$ for $i\neq j$. 
\end{definition}
\begin{remark}
    If $B$ is a point, then $\Gamma_{E,B}^n=\Gamma_E^n$ and $dE_B=dE$. Hence,
    $\vv{\TC}_n[p\colon E\to B]= \vv{\TC}_n(E)$.
\end{remark}
The following proposition shows that the sequential directed parametrized topological complexity of a \(d\)-fibration is nondecreasing with respect to \(n\), analogous to the case of sequential parametrized topological complexity \cite[Lemma 5.3]{Farber-Paul1}.
\begin{proposition} Let $p\colon E \rightarrow B$ be a d-fibration, with an initial point $e_0 \in E$. Then
    $$\vv{\TC}_n[p\colon E\to B]\leq \vv{\TC}_{n+1}[p\colon E\to B].$$
\end{proposition}
\begin{proof}
    Suppose the $d$-space $(E,dE)$ has an initial point $e_0$. If $\vv{\TC}_{n+1}[p\colon E\to B]=k$, then $\Gamma_{E, B}^{n+1}$ can be covered by $k$ ENRs $F_1,F_2,\dots,F_k$ such that for each $i=1,2,\dots,k$, there exists a continuous section $s_i\colon F_i\to dE_B$ of the map $\vv{\Pi}_{n+1}$. Now for $1\leq i\leq k$ define ENRs
    \[\widetilde{F_i}:=\{(e_1,e_2,\dots,e_n)\mid (e_0,e_1,\dots,e_n)\in F_i\}\]
    and 
    \[s_i'\colon \widetilde{F_i}\to dE_B \text{ by } s_i'(e_1,e_2,\dots,e_n):=s_i(e_0,e_1,\dots,e_n).\]

    Then the collection $\{\widetilde{F_i}\mid 1\leq i\leq k\}$ covers $\Gamma_{E,B}^n$ and $s_i'$ is a continuous section $\vv{\Pi}_n$ for each $1\leq i\leq k$. Thus,
    $\vv{\TC}_n[p\colon E\to B]\leq \vv{\TC}_{n+1}[p\colon E\to B]$.
\end{proof}
The next proposition shows that $\vv{\TC}_n$ does not increase under base restriction.
\begin{proposition}
    Let $p\colon E\to B$ be a $d$-fibration and $B'\subseteq B$. If $p'\colon E'=p^{-1}(B')\to B'$ is the restricted $d$-fibration, then 
$\vv{\TC}_n[p\colon E'\to B']\leq \vv{\TC}_n[p\colon E\to B]$.

    In particular, if $F=p^{-1}(\{b\})$, then %$\vv{\TC}_n(F)\leq \vv{\TC}_n[p\colon E\to B]$.
    \begin{equation}\label{Fibre_TC}
        \vv{\TC}_n(F)\leq \vv{\TC}_n[p\colon E\to B].
    \end{equation}
\end{proposition}
\begin{proof}
    Let $U\subseteq \Gamma_{E,B}^{n}$ be an ENR such that there exists a section $s\colon U\to dE_B$ of $\vv{\Pi}_n$. 
    We now consider the ENR $V:= \Gamma_{E', B'}^n\cap U$. 
    We first show that $s(V)\in dE'_{B'}$. 
    Let $(x_1,\dots,x_n)\in V$. 
    Then $s(x_1,\dots,x_n)\in dE_B$. 
    Hence $p\circ (s(x_1,\dots,x_n))=b$ for some $b\in B$ and for all $t\in \vv{I}$. 
    Moreover note that $p\circ (s(x_1,\dots,x_n))=b=p(s(x_1,\dots,x_n))(0)$. 
    Since $x\in E'$, we have $b\in B'$. Therefore $s(x_1,\dots,x_n)(t)\in E'$ for all $t\in \vv{I}$. This shows that $s(V)\subseteq dE'_B$. Then we define $s'\colon V\to dE_B'$ by $s':=s\vert_V$. 
    This gives the desired section of $\vv{\Pi}_n'\colon dE_B'\to \Gamma_{E',B'}^n$. using the similar argument for the cover $\Gamma_{E,B}^n$ by ENRs, we obtain the first inequality.

    The second inequality follows from the first inequality by setting $B'=\{b\}$.
\end{proof}
%We show that sequential DTC of a fibre is less than the sequential DPTC of a d-fibration.
%\begin{proposition}\label{Fibre_TC}
 %   $\vv{\TC}_n[p\colon E\to B]\geq \vv{\TC}_n(F)$, where $F$ is the fibre of the $d$-fibration $p\colon E\to B$.
%\end{proposition}
%\begin{proof}
 %    Suppose $\vv{\TC}_n[p\colon E\to B]=k$. Then $\Gamma_{E,B}^n$ is covered by $k$ ENRs $U_1$,$U_2$,\dots,$U_k$ with each ENR $U_i$ admits continuous section $s_i$ of the map $\vv{\Pi}_n$. Let $F=p^{-1}(b)$ for some $b\in B$. Then $\tilde{U_i}= U_i\cap \Gamma_{E,b}^n$ is an ENR. The $k$ ENRs $\tilde{U_1},\tilde{U_2},\dots,\tilde{U_k}$ cover $\Gamma_{E,b}^n=\Gamma_{X}^n$. Also, note that $s_i'= s_i\vert_{\tilde{U_i}}\colon \tilde{U}_i\to dE_b= dF$ is a continuous section of $\vv{\pi}_n = \vv{\Pi}_n\vert_{dE_b}\colon dF= dE_{b}\to E_b^n=F^n$. Therefore,
  %  \[\vv{\TC}_n(F)\leq \vv{\TC}_n[p\colon E\to B].\]
%\end{proof}
The following theorem is a sequential analogue of \cite[Theorem~4.6]{SDNDAS2025}, characterizing the condition $\vv{\TC}_n[p\colon E\to B]=1$ via the $n$-basic dihomotopy type of the fibre.
\begin{theorem} \label{dicontractible fibre}
    Let $p\colon E\to B$ be a $d$-fibration with fibre $F$ being either a contractible $d$-space or having an initial point. Then $F$ is $n$-basic dihomotopy equivalent to a point if and only if \[\vv{\TC}_n[p\colon E\to B]=1.\]
\end{theorem}
\begin{proof}
    Suppose \( \vv{\TC}_n[p \colon E \to B] = 1 \). Then \eqref{Fibre_TC} implies that \( \vv{\TC}_n(F) = 1 \), so by the definition of \( \vv{\TC}_n(F) \), there exists a continuous section \( s \colon \Gamma_F^n \to dF \) of the \( n \)-th \( d \)-paths map \( \vv{\pi}_n \colon dF \to \Gamma_F^n\subseteq F^n \). Applying Proposition~\ref{Important_Result}, we conclude that \( F \) is \( n \)-basic dihomotopy equivalent to a point.

    Conversely, suppose $F$ is $n$-basic dihomotopy equivalent to a point. Then, by Proposition \ref{Important_Result}, there exists a global continuous section $s\colon \Gamma_F^n\to dF$ of the $n$-th $d$-paths map $\vv{\pi}_n\colon dF\to \Gamma_{F}^n$. Let $(e_1,e_2,\dots,e_n)\in \Gamma_{E,B}^n$, so that $p(e_1)=p(e_2)=\dots=p(e_n)=b$ for some $b\in B$. Set $F_b:= p^{-1}(b)$. Since all fibres of the $d$-fibration $p\colon E\to B$ are $d$-homotopic, there exists a $d$-homotopy $H_b\colon F_b\times \vv{I}\to F$. Now, define a $d$-map $\beta\colon \Gamma_{E,B}^n\to \Gamma_F^n$ by
    \[\beta\left((e_1,e_2,\dots,e_n)\right):= \left(H_b(e_1,1),H_b(e_2,1),\dots,H_b(e_n,1)\right).\] Since $\left(e_1,e_2,\dots,e_n\right)\in \Gamma_{F_b}^n$, it follows that $\left(H_b(e_1,1),H_b(e_2,1),\dots,H_b(e_n,1)\right)\in \Gamma_F^n$. Hence $\beta$ is a well-defined map. We note that the map $\beta\colon \Gamma_{E,B}^n\to \Gamma_F^n$ fits into the following commutative diagram:
    \[\begin{tikzcd}
	dF & {dE_B} \\
	{\Gamma_F^n} & {\Gamma_{E,B}^n.}
	\arrow["\alpha", hook, from=1-1, to=1-2]
	\arrow["{\vv{\pi}_n}", from=1-1, to=2-1]
	\arrow["{\vv{\Pi}_n}", from=1-2, to=2-2]
	\arrow["s", bend left, from=2-1, to=1-1]
	\arrow[hook, from=2-1, to=2-2]
	\arrow["\beta", bend left, dashed, from=2-2, to=2-1]
\end{tikzcd}\]
   The composition $\alpha\circ s\circ\beta$ is clearly continuous on $\Gamma_{E,B}^n$ and satisfies $\vv{\Pi}_n\circ\left(\alpha\circ s\circ\beta\right)= Id$. Hence, $\vv{\TC}_n[p\colon E\to B]=1$.
\end{proof}
\begin{corollary}
      Let $p\colon E\to B$ be a $d$-fibration with fibre $F$ being either a contractible $d$-space or having an initial point. If $\vv{\TC}_n[p\colon E\to B]=1$, then $F$ is dicontractible.
\end{corollary}
\begin{proof}
   Suppose \(\vv{\TC}_n[p\colon E \to B] = 1\). Then, by the above proposition, we have \(\vv{\TC}_n(F) = 1\). Applying Proposition~\ref{Properties_higher_TC}~(c), it follows that \(\vv{\TC}_2(F) = 1\), which in turn implies \(\vv{Cat}(F) = 1\) by \cite[Proposition 3.5 {(1)}]{SDNDAS2025}. From \cite[Theorem 1]{EGMFAS2020}, it follows that $F$ is dicontractible.
\end{proof}
The following proposition shows that the inequality in~\eqref{Fibre_TC} becomes an equality for trivial \(d\)-fibrations.
\begin{proposition}\label{directed_TC_trivial_fibration}
    Let $p\colon E\to B$ be a trivial $d$-fibration with fibre $F$. Then \[\vv{\TC}_n[p\colon E\to B]= \vv{\TC}_n(F).\]
\end{proposition}
\begin{proof}
Without loss of generality, let us take \( E = B \times F \), with the projection maps \( p = pr_1\colon B \times F \to B \) and \( pr_2\colon B \times F \to F \). 
Under this identification, we obtain canonical \(d\)-homeomorphisms
$\varphi\colon \Gamma_{E,B}^n \xrightarrow{\cong} \Gamma_F^n\times B
\text{ and } 
\psi\colon dE_B \xrightarrow{\cong} dF\times B$
given by
$\varphi(\tilde e_1,\tilde e_2,\ldots,\tilde e_n)
=
((e_1,e_2,\ldots,e_n),b)$,
where \(p(\tilde e_i)=b\) and \(pr_2(\tilde e_i)=e_i\) for each \(i\), and
$\psi(\tilde\gamma)=(\gamma,b)$,
where \(p(\tilde\gamma(t))=b\) and \(pr_2(\tilde\gamma(t))=\gamma(t)\in F\) for all \(t\in \vv I\).
%Under this identification, we describe the following canonical \( d \)-homeomorphisms:
%\[
%\varphi\colon \Gamma_{E, B}^n \xrightarrow{\cong} \Gamma_{F}^n \times B, \quad
%\psi\colon dE_B \xrightarrow{\cong} dF \times B.
%\]

%We define \( \varphi\colon \Gamma_{E, B}^n \to \Gamma_{F}^n \times B \) by
%\[
%\varphi(\tilde{e}_1, \tilde{e}_2, \ldots, \tilde{e}_n) = ((e_1, e_2, \ldots, e_n), b),
%\]
%where \( p(\tilde{e}_i) = b \) and \( pr_2(\tilde{e}_i) = e_i \) for all \( i = 1, 2, \ldots, n \).

%Analogously, define \( \psi\colon dE_B \to dF \times B \) by
%\[
%\psi(\tilde{\gamma}) = (\gamma, b),
%\]
%where \( p(\tilde{\gamma}(t)) = b \) and \( pr_2(\tilde{\gamma}(t)) = \gamma(t) \in F \) for all \( t \in \vv{I} \).

These $d$-homeomorphisms give rise to the following commutative diagram:
\[\begin{tikzcd}
	{dE_B} & {dF\times B} \\
	{\Gamma_{E,B}^n} & {\Gamma_{F}^n\times B.}
	\arrow["\psi", from=1-1, to=1-2]
	\arrow["{\vv{\Pi}_n}"', from=1-1, to=2-1]
	\arrow["{\vv{\pi}_n\times Id_B}", from=1-2, to=2-2]
	\arrow["\varphi"', from=2-1, to=2-2]
\end{tikzcd}\]
Note from \eqref{Fibre_TC} that $\vv{\TC}_n(F)\leq \vv{\TC}_n[p\colon E\to B]$.
   To prove the reverse inequality, we assume $\vv{\TC}_n(F)=k$. Then $\Gamma_{F}^n$ can be partitioned into $k$ ENRs $F_1,F_2,\dots,F_k$ and sections $s_i$ over $F_i$ of the map $\vv{\pi}_n\colon dF\to \Gamma_{F}^n$ for each $1\leq i\leq k$. Now, for $1\leq i\leq k$, define
   \[G_i:= \varphi^{-1}(F_i\times B) \text{ and } s_i':= \psi^{-1}\circ (s_i\times Id)\circ \varphi.\]
   Clearly, the ENRs $G_1,G_2,\dots,G_k$ cover $\Gamma_{E,B}^n$ and each $s_i'$ defines continuous section over $G_i$ of the map $\vv{\Pi}_n$. This implies that
   $\vv{\TC}_n[p\colon E\to B]\leq \vv{\TC}_n(F)$.
\end{proof}
We now compare sequential directed parametrized topological complexity and sequential parametrized topological complexity. 
The following example shows that the latter can be strictly greater than the former. 
\begin{example}
\normalfont{Let \(pr_1\colon B\times X\to B\) denote the projection onto the first factor, where \(X\) is as in Proposition~\ref{TC(X)} and \(B\) is an arbitrary \(d\)-space. Then by Proposition~\ref{directed_TC_trivial_fibration} and \cite[Example 3.2]{Farber-Paul1}, we have $\vv{\TC}_n[pr_1\colon B\times X\to B]=1$, and $\TC_n[pr_1\colon B\times X\to B]=n$. Thus \[\vv{\TC}[pr_1\colon B\times X\to B]< \TC_n[pr_1\colon B\times X\to B].\]}
  \end{example}
In contrast to the previous example, the following proposition gives the reverse inequality under the assumption that the fiber of the $d$-fibration is strongly connected.
\begin{proposition}\label{comparision: directed and usual}
  Suppose the fiber $F$ of a $d$-fibration $p\colon E\to B$ is strongly connected. Then \[\TC_n[p\colon E\to B]\leq \vv{\TC}_n[p\colon E\to B].\]   
\end{proposition}
We first prove the following lemma, which will be used to prove Proposition~\ref{comparision: directed and usual}.
%\begin{lemma}\label{Gamma_B^n_for_SC}
 %   Let $p\colon E\to B$ be a $d$-fibration. Then $\Gamma_{E,B}^n= E_B^n$, if the fiber $F$ of $p$ is strongly connected.
%\end{lemma}
\begin{lemma}\label{Gamma_B^n_for_SC}
Let $p\colon E\to B$ be a $d$-fibration with strongly connected fiber $F$. Then $\Gamma_{E,B}^n = E_B^n$.
\end{lemma}
\begin{proof}
The inclusion \( \Gamma_{E,B}^n \subseteq E_B^n \) is immediate from the definition. To prove the reverse inclusion, let 
$(e_1, e_2, \ldots, e_{n}) \in E_B^n.$
Then \( p(e_1) = p(e_2) = \cdots = p(e_{n}) \). Suppose $p(e_i)=b$ for all $i=1,\dots,n$.
Since \( p^{-1}(b) \) is strongly connected, as established in the proof of Proposition~\ref{Comparision_TC_DirectedTC}, there exists a directed path \( \gamma \) in \( p^{-1}(b) \) such that 
$\gamma\left( \tfrac{k}{n-1} \right) = e_{k+1}$, for all $k = 0, 1, \ldots, n-1$.
As all fibers are \( d \)-homotopic to each other, it follows from \cite[Lemma 4.11 (1)]{SDNDAS2025} that \( (e_1, e_2, \ldots, e_{n}) \in \Gamma_{E,B}^n \), which completes the proof.
\end{proof}
\begin{remark}
    Lemma~\ref{Gamma_B^n_for_SC} recovers \cite[Lemma 4.11 (2)]{SDNDAS2025}.
\end{remark}
%The following result relates the $n$-directed parametrized topological complexity to the $n$-th sequential parametrized topological complexity.
\begin{proof}[Proof of Proposition~\ref{comparision: directed and usual}]
    Since $F$ is strongly connected, $\Gamma_{E,B}^n= E_B^n$ by Lemma \ref{Gamma_B^n_for_SC}. Hence, any local continuous section of $\vv{\Pi}_n\colon dE_B\to \Gamma_{E,B}^n$ can be viewed as a local continuous section of $\Pi_n\colon E_B^I\to E_B^n$. This completes the proof.
\end{proof}
We now give a $d$-fibration for which the inequality in Proposition~\ref{comparision: directed and usual} is strict.
\begin{example}
   \normalfont{Let $pr_1\colon B\times X\to B$ be the projection map onto the first factor, where $X$ is the space given in Example~\ref{directed_TC_D^2} and $B$ is a space with any $d$-structure. 
Then by Proposition~\ref{directed_TC_trivial_fibration}, we have $\vv{\TC}_n[pr_1\colon B\times X\to B]=\vv{\TC}_n(X)\geq n$, while \cite[Example 3.2]{Farber-Paul1} implies that $\TC_n[pr_1\colon B\times X\to B]=\TC_n(X)=1$.}
\end{example}

\subsection{Sequential analogue of regular d-fibrations} \label{Sequential analogue of regular d-fibrations}
Next, we establish the product inequality for the sequential version of directed parametrized topological complexity. To do so, we begin by extending the notion of \( d \)-regularity for a \( d \)-fibration, as introduced in \cite[Definition 4.13]{SDNDAS2025}.
\begin{definition}\label{n-regular d fibration}
A $d$-fibration $p\colon E\to B$ is called \emph{$n$-regular} if there exists a partition
$\Gamma_{E,B}^n= \bigcup_{i=1}^k A_i$
such that each $A_i$ is an ENR, the map $\vv{\Pi}_n\colon dE_B\to \Gamma_{E,B}^n$ admits a continuous section over each $A_i$, and the unions $A_1\cup A_2\cup\dots\cup A_r$ are closed for all $1\leq r\leq k$.
\end{definition}
%\begin{definition}\label{n-regular d fibration}
 %   A $d$-fibration $p\colon E\to B$ is called \emph{$n$-regular} if one can find a partition 
  %  \[\Gamma_{E,B}^n= \bigcup_{i=1}^k A_i,~~ k=\vv{\TC}_n[p\colon E\to B]\]
   % into ENRs such that the map $\vv{\Pi}_n\colon dE_B\to \Gamma_{E,B}^n$ admits a continuous section over each $A_i$ and, additionally, the finite unions $A_1\cup A_2\cup\dots\cup A_r$ are closed for all $1\leq r\leq k$. 
%\end{definition}
Observe that when $B = \{*\}$, the notion of an $n$-th $d$-regular $d$-fibration $p\colon E \to B$ reduces to the notion of $n$-regular $d$-space as defined in Definition~\ref{n-regular d-space}.

\begin{remark}\label{consequence of n regular d fibration}
    For an \( n \)-regular \( d \)-map with ENRs \( \{A_i\}_{i=1}^k \), the sets satisfy the condition  
\[
\Bar{A_i} \cap A_j = \emptyset \quad \text{for } i < j.
\]
\end{remark}
Let $p\colon E\to B$ and $p'\colon E'\to B'$ be two $d$-fibrations. Note that any path $\gamma\colon [0,1]\to (E\times E')_{B\times B'}$ has the form $\gamma(t)=(\gamma_E(t), \gamma_{E'}(t))$ and we declare $\gamma$ to be directed if both its coordinates are directed. i.e., $\gamma_E\in dE_B$ and $\gamma_{E'}\in dE'_{B'}$. Consequently, $\Gamma_{E\times E',\,B\times B'}^n$ is determined componentwise by $\Gamma_{E,B}^n$ and $\Gamma_{E',B'}^n$, leading to the following lemma.
\begin{lemma}\label{Gamma_for_product}
    Let $p\colon E\to B$ and $p'\colon E'\to B'$ be two $d$-fibrations. Then
    $$\Gamma_{E\times E', B\times B'}^n= \Gamma_{E,B}^n \times \Gamma_{E',B'}^n.$$
\end{lemma}
\begin{proof}
    Define
    $\Phi\colon \Gamma_{E\times E', B\times B'}^n \to \Gamma_{E,B}^n \times \Gamma_{E',B'}^n$ by
    $$\Phi\left((e_1,e_1'),(e_2,e_2'),\dots,(e_n,e_n')\right):=\left((e_1,e_2,\dots,e_n),(e_1',e_2',\dots,e_n')\right).$$

    If $\left((e_1,e_1'),(e_2,e_2'),\dots,(e_n,e_n')\right)\in \Gamma_{E\times E', B\times B'}^n$, then 
    $(p\times p')(e_1,e_1')=\cdots= (p\times p')(e_n,e_n')$
    and there exists $\gamma\in d(E\times E')_{B\times B'}$ such that $\gamma(\tfrac{k}{n-1})=(e_{k+1},e_{k+1}')$ for $i=0,1,\dots,n-1$.
    Clearly, $pr_1\circ \gamma\in dE_B$ and $pr_2\circ \gamma\in dE'_{B'}$ satisfying $(pr_1\circ\gamma)(\tfrac{k}{n-1})=e_{k+1}$ and $(pr_2\circ\gamma)(\tfrac{k}{n-1})=e_{k+1}'$ for $i=0,1,\dots,n-1$.
    This implies that $\left((e_1,e_2,\dots,e_n),(e_1',e_2',\dots,e_n')\right)\in \Gamma_{E,B}^n \times \Gamma_{E',B'}^n$. Therefore, $\Phi$ is a well-defined map.

Suppose $\bigl((e_1,\ldots,e_n),(e_1',\ldots,e_n')\bigr)\in
\Gamma_{E,B}^n\times\Gamma_{E',B'}^n$. Then there exist
$\gamma\in dE_B$ and $\gamma'\in dE'_{B'}$ such that
$\gamma(\tfrac{k}{n-1})=e_{k+1}$,
$\gamma'(\tfrac{k}{n-1})=e'_{k+1}$ for $0\le k\le n-1$, with
$p(e_1)=\cdots=p(e_n)$ and $p(e_1')=\cdots=p(e_n')$.
%Suppose $\left((e_1,e_2,\dots,e_n),(e_1',e_2',\dots,e_n')\right)\in \Gamma_{E,B}^n \times \Gamma_{E',B'}^n$. Then there exist $\gamma\in dE_B$, $\gamma'\in dE'_{B'}$ such that $\gamma(\tfrac{k}{n-1})=e_{k+1}$, $\gamma'(\tfrac{k}{n-1})=e_{k+1}'$ for $i=0,1,\dots,n-1$ and $p(e_1)=\cdots=p(e_n)$, $p(e_1')=\cdots=p(e_n')$. and
     %\[\exists ~ \gamma'\in dE'_{B'} \text{ such that } \gamma'(\tfrac{k}{n-1})=e_{k+1}' \text{ for } i=0,1,\dots,n-1 \text{ and } p(e_1')=\cdots=p(e_n').\]
     Note that $\gamma\times \gamma'\in d(E\times E')_{B\times B'}$ with $(\gamma\times \gamma')(\frac{k}{n-1})=(e_{k+1},e_{k+1}')$ for $k=0,1,\dots,n-1$ and $(p\times p')(e_1,e_1')=\cdots=(p\times p')(e_n,e_n')$. Hence, 
     $\left((e_1,e_1'),(e_2,e_2'),\dots,(e_n,e_n')\right)\in \Gamma_{E\times E', B\times B'}^n.$
     Thus, the map $\Phi$ is surjective. Similarly, we can show that the map $\Phi$ is also injective.
 Therefore, $\Gamma_{E\times E', B\times B'}^n= \Gamma_{E,B}^n \times \Gamma_{E',B'}^n$.
\end{proof}
We now establish the product inequality for sequential directed parametrized topological complexity $\vv{\TC}_n [-]$.
\begin{proposition} \label{product inequality_parametrized}
     Let $p\colon E\to B$ and $p'\colon E'\to B'$ be two $n$-regular $d$-fibrations. Then
     \[\vv{\TC}_n[p\times p'\colon E\times E'\to B\times B']\leq \vv{\TC}_n[p\colon E\to B]+ \vv{\TC}_n[p'\colon E'\to B']-1.\]
\end{proposition}
\begin{proof}
    Let \( \vv{\TC}_n[p\colon E \to B] = k \) and \( \vv{\TC}_n[p'\colon E' \to B'] = l \), and suppose we have partitions
$\Gamma_{E,B}^n = \bigcup_{i=1}^k A_i$ and $\Gamma_{E',B'}^n = \bigcup_{i=1}^l B_i$
as described in Definition~\ref{n-regular d fibration}. That is, each \( A_i \) and \( B_i \) is an ENR, and the maps
$\vv{\Pi}_n^B\colon dE_B \to \Gamma_{E,B}^n$ and $\vv{\Pi}_n^{B'}\colon dE'_{B'} \to \Gamma_{E',B'}^n$
admit continuous sections \( s_i \) over \( A_i \) and \( s_i' \) over \( B_i \), respectively. Moreover, Since both $d$-fibrations $p\colon E\to B$ and $p'\colon E'\to B'$ are $n$-regular, for each \( r = 1,2,\dots,k \), the union \( \bigcup_{i=1}^r A_i \) is closed, and similarly, for each \( r = 1,2,\dots,l \), the union \( \bigcup_{i=1}^r B_i \) is closed.

Consider the sets $\bigcup_{i+j=t} (A_i\times B_j)=G_t\subset \Gamma_{E\times E'. B\times B'}^n$,
%\begin{equation}\label{union}
%\bigcup_{i+j=t} (A_i\times B_j)=G_t\subset \Gamma_{E\times E'. B\times B'}^n,    
%\end{equation}
where $t=2,\dots,k+l$. clearly, $\{G_t\}_{t=2}^{k+l}$ are pairwise disjoint and form an ENR over of $\Gamma_{E \times E', B \times B'}^n $ by Lemma~\ref{Gamma_for_product}.
%Using Lemma \ref{Gamma_for_product}, we note that the sets \( A_i \times B_j \), where \( 1 \leq i \leq k \) and \( 1 \leq j \leq l \), form an ENR partition of \( \Gamma_{E \times E', B \times B'}^n \).
Then continuous sections $s_i\colon A_i\to dE_B$ and $s_j'\colon B_j\to dE'_{B'}$ obviously produce continuous sections $\sigma_{i,j}:= s_i\times s_j'\colon A_i\times B_j\to d(E\times E')_{B\times B'}$ of $\vv{\Pi}_n^B\times \vv{\Pi}_n^{B'}$ over $A_i\times B_j$.
 Now by Remark~\ref{consequence of n regular d fibration}, the collection
$\sqcup_{i+j=t} \sigma_{i,j}\colon G_t\to d(E\times E')$ gives a continuous section of $\vv{\Pi}_n\colon d(E\times E')\to (E\times E')_{B\times B'}^n$.
%We observe that the terms of the union \eqref{union} are pairwise disjoint and open in $G_t$ (due to Remark \ref{consequence of n regular d fibration}) and hence the collection of continuous maps $\sigma_{i,j}$ defines a continuous section $G_t\to d E\times E'_{B\times B'}$. 
Consequently, we obtain 
$\vv{\TC}_n[p\times p'\colon E\times E'\to B\times B']\leq k+l-1$.
\end{proof}
We now compare usual parametrized topological complexity with the directed LS category in the following proposition.
\begin{proposition}
    Let $p\colon E\to B$ be a fibration where $E_B^r$ is a directed space with an initial point. Then 
    $\TC_r[p\colon E\to B]\leq \vv{\ct}(E_B^r)$.
\end{proposition}
\begin{proof}
Note that we have
$\TC_r[p\colon E\to B]\leq \ct(E_B^r).$
Since \(\ct(X)\leq \vv{\ct}(X)\) for every directed space \(X\) with an initial point \cite[Proposition~3.4]{SDNDAS2025}, the desired inequality follows immediately.
\end{proof}

\subsection{Fibrewise \texorpdfstring{$n$}{}-basic dihomotopy equivalence} \label{fibrewise n-dihomotopy equivalence}
We now introduce the notion of fibrewise \( n \)-basic dihomotopy equivalence, under which we show that the sequential version of directed parametrized topological complexity is an invariant. 

In \cite{Farber-Paul1}, it was shown that sequential parametrized topological complexity is a fibrewise homotopy invariant. The following example shows that the analogous statement does not hold in the directed setting.

\begin{example}
\normalfont{Let $p_1\colon X\times B\to B$ and $p_2\colon \vv{\mathbb{S}^1}\times B\to B$ denote the projection maps onto \(B\), where \(\vv{\mathbb{S}^1}\) is the directed circle described in Example~\ref{directed_circle}, \(X\) is the directed space considered in Proposition~\ref{TC(X)}, and \(B\) is an arbitrary directed space. Then \(p_1\) and \(p_2\) are fibrewise homotopy equivalent.
Using Propositions~\ref{directed_TC_trivial_fibration} and~\ref{TC(X)}, together with Example~\ref{directed_circle}, we obtain
 $\vv{\TC}_n[p_1\colon X\times B\to B]=1$ while $\vv{\TC}_n[p_2\colon \vv{\mathbb{S}^1}\times B\to B]=n$.
Hence sequential directed parametrized topological complexity is not a fibrewise homotopy invariant.}
\end{example}
This motivates the following definition of fibrewise n-basic dihomotopy equivalence between two $d$-fibrations. We begin by recalling the following definitions.
\begin{definition}{\cite{SDNDAS2025}}
Let \( p\colon E \to B \) and \( p'\colon E' \to B \) be two \( d \)-fibrations.
\begin{itemize}
    \item[(a)] A \emph{fibrewise $d$-map} from \( p\colon E \to B \) to \( p'\colon E' \to B \) is a $d$-map \( f\colon E \to E' \) such that \( p' \circ f = p \).
    
    \item[(b)] A \emph{fibrewise \( d \)-homotopy} is a \( d \)-map \( F\colon E \times \vv{I} \to E' \) such that \( p'\big(F(-,t)\big) = p \) for all \( t \in \vv{I} \). In particular, \( F \) is a \( d \)-homotopy between the fibrewise maps \( F(-,0) \) and \( F(-,1) \).
\end{itemize}
\end{definition}
\begin{definition}\label{fibrewise n-basic dihomotopy equivalence}
Two $d$-fibrations $p\colon E\to B$ and $p'\colon E'\to B$ are said to be \emph{fibrewise $n$-basic dihomotopy equivalent} if the following conditions hold:
\begin{enumerate}
    \item There exist fibrewise $d$-maps $f\colon E\to E'$ and $g\colon E'\to E$ such that there are fibrewise $d$-homotopies from $f\circ g$ to $Id_{E'}$ and from $g\circ f$ to $Id_E$.
    
    \item There exists a continuously graded map \( F\colon dE' \to dE \) fitting into the following commutative diagram:
    \[
    \begin{tikzcd}
    dE && dE' \\
    & dB
    \arrow["df", bend left=20, from=1-1, to=1-3]
    \arrow["dp"', from=1-1, to=2-2]
    \arrow["F", bend left=20, from=1-3, to=1-1]
    \arrow["dp'", from=1-3, to=2-2]
    \end{tikzcd}
    \]
    such that for every \( (e_1, e_2, \dots, e_n) \in \Gamma_{E,B}^n \), the map
    $F_{e_1, \dots, e_n} \colon dE'\big(f(e_1), \dots, f(e_n)\big) \to dE(e_1, \dots, e_n)$
    is a homotopy equivalence, with homotopy inverse given by \( df_{e_1, \dots, e_n} \).

    \item There exists a continuously graded map \( G\colon dE \to dE' \) fitting into the following commutative diagram:
    \[
    \begin{tikzcd}
    dE' && dE \\
    & dB
    \arrow["dg", bend left=20, from=1-1, to=1-3]
    \arrow["dp'"', from=1-1, to=2-2]
    \arrow["G", bend left=20, from=1-3, to=1-1]
    \arrow["dp", from=1-3, to=2-2]
    \end{tikzcd}
    \]
    such that for every \( (e_1', e_2', \dots, e_n') \in \Gamma_{E',B}^n \), the map
    $G_{e_1', \dots, e_n'} \colon dE\big(g(e_1'), \dots, g(e_n')\big) \to dE'(e_1', \dots, e_n')$
    is a homotopy equivalence, with homotopy inverse given by \( dg_{e_1', \dots, e_n'} \).
\end{enumerate}
\end{definition}
\begin{theorem} %\label{n-dihomotopy equi}
    If $d$-fibrations $p\colon E\to B$ and $p'\colon E'\to B$ are fibrewise $n$-basic dihomotopy equivalent, then 
    \[\vv{\TC}_n[p\colon E\to B]= \vv{\TC}_n[p'\colon E'\to B].\]
\end{theorem}
\begin{proof}
    Suppose $p\colon E\to B$ and $p'\colon E'\to B$ are fibrewise $n$-basic dihomotopy equivalent, we have $d$-maps $f\colon E\to E'$ and $g\colon E'\to E$ which are $d$-homotopy inverses of each other satisfying the following commutative diagram:
\[\begin{tikzcd}
	E && {E'} \\
	& {B.}
	\arrow["f", bend left=20, from=1-1, to=1-3]
	\arrow["p"', from=1-1, to=2-2]
	\arrow["g", shift right, bend left=20, from=1-3, to=1-1]
	\arrow["{p'}", from=1-3, to=2-2]
\end{tikzcd}\]

Additionally, there is a $d$-map $G\colon dE \to dE'$ satisfying condition (3) of Definition \ref{fibrewise n-basic dihomotopy equivalence}. We claim that $G$ restricts to a map from $dE_B$ onto $dE'_B$. Let $\gamma \in dE_B$. Then $(dp' \circ G)(\gamma) = dp(\gamma)$, which implies $p'(G(\gamma)(t)) = p(\gamma(t))$ for all $t \in \vv{I}$. Hence $G(\gamma) \in dE'_B$. We continue to denote the restriction of $G$ to $dE_B$ by $G$. This yields the following commutative diagram:
\[\begin{tikzcd}
	{dE'_{B}} & {dE_B} & {dE'_{B}} \\
	{\Gamma_{E',B}^n} & {\Gamma_{E,B}^n} & {\Gamma_{E',B}^n,}
	\arrow["dg", from=1-1, to=1-2]
	\arrow["{\vv{\Pi'}_n}"', from=1-1, to=2-1]
	\arrow["G", from=1-2, to=1-3]
	\arrow["{\vv{\Pi}_n}", from=1-2, to=2-2]
	\arrow["{\vv{\Pi'}_n}", from=1-3, to=2-3]
	\arrow["{\bar{g}}"', from=2-1, to=2-2]
	\arrow["{\bar{f}}"', from=2-2, to=2-3]
\end{tikzcd}\] where the $d$-map $\bar{g}$ is defined as $\bar{g}:= {\overbrace{(g \times g \times \cdots \times g)}^{n\ \text{times}}}\vert_{\Gamma_{E',B}^n}$ and $\bar{f}$ is also defined similarly. Note that these maps are well-defined as they are fibrewise maps.

Let $(e_1',e_2',\dots,e_n')\in \Gamma_{E',B}^n$. Then $p'(e_1')=p'(e_2')=\cdots=p'(e_n')$. Suppose $p'(e_i')=b$ for all $i=1,2,\dots,n$. Therefore $p(g(e_i'))=b$ for all $i=1,2,\dots,n$. This implies that $(g(e_1'),g(e_2'),\dots,g(e_n'))\in \Gamma_{E,B}^n$.

Let $U\subseteq \Gamma_{E,B}^n$ be an ENR such that there is a continuous section $s$ of $\vv{\Pi}_n$ on $U$. Define $V:= \bar{g}^{-1}(U)$ and $s'\colon V\to dE'_B$ by $s'(v):= G \circ s\circ \bar{g}\vert_{V}(v)$ for $v\in V$. 
For $(e_1',e_2',\dots,e_n')\in \Gamma_{E',B}^n$, the path $\gamma:=s(g(e_1'),g(e_2'),\dots,g(e_n'))\in dE_B$, i.e., $\gamma\in dE_B(g(e_1'),g(e_2'),\dots,g(e_n'))$. Since $G(\gamma)\in dE'_B(e_1',e_2',\dots,e_n')$, we get $\vv{\Pi'}_n(G(\gamma))=(e_1',e_2',\dots,e_n')$. Since $s$ is continuous on $U$, $\bar{g}$ is continuous on $\Gamma_{E',B}^n$, $G$ is continuous and graded, $s'$ is continuous on $V$. Therefore $\vv{\TC}_n[p'\colon E'\to B]\leq \vv{\TC}_n[p\colon E\to B]$.

  Using the $d$-map $F\colon dE' \to dE$ satisfying condition~(2) of Definition~\ref{fibrewise n-basic dihomotopy equivalence} and following the same line of reasoning, we obtain the reverse inequality. This completes the proof.
\end{proof}
Next, we show that equality holds in Proposition~\ref{comparision: directed and usual} for a particular directed structure introduced in \cite[Definition 5.1]{SDNDAS2025}. Suppose \(p\colon E\to B\) is a fibration, where \((B,dB)\) is a directed space. We define a directed structure on \(E\) by
\[
\gamma\in dE \quad \text{if and only if} \quad p\circ \gamma\in dB.
\]
With this directed structure, \(p\colon E\to B\) becomes a \(d\)-fibration. Moreover, the induced directed structure on the fibre \(F\) of \(p\) makes \(F\) a strongly connected directed space; see \cite[Remark 5.2]{SDNDAS2025}.
\begin{proposition}\label{equality}
Let $p\colon E\to B$ be a fibration, where $(B,dB)$ is a directed space, and let $(E,dE)$ be the directed space defined above. Then
\[
\vv{\TC}_n[p\colon E\to B]=\TC_n[p\colon E\to B].
\]
\end{proposition}
\begin{proof}
Since \(F\) is strongly connected, Proposition~\ref{comparision: directed and usual} reduces the proof to showing that $\vv{\TC}_n[p\colon E\to B]\leq \TC_n[p\colon E\to B]$.

Suppose \(\TC_n[p\colon E\to B]=k\). Since \(E\) and \(B\) are ENRs, the space \(E_B^n\) admits a partition by ENRs \(\{U_1,\dots,U_k\}\) such that, for each \(i\), there exists a section $s_i\colon U_i\to E_B^I$ of the evaluation map $\Pi_n\colon E_B^I\to E_B^n$.

Let \((e_1,\dots,e_n)\in U_i\). Then $p(e_1)=\cdots=p(e_n)=b$ for some \(b\in B\). Hence $s_i(e_1,\dots,e_n)\in F_b^I$, where \(F_b=p^{-1}(b)\). Since the induced \(d\)-structure on \(F\) is given by the free path space, it follows that $s_i(e_1,\dots,e_n)\in dE_B$.
Consequently, each \(s_i\) may be regarded as a section of the directed evaluation map $\vv{\Pi}_n\colon dE_B\to \Gamma_{E,B}^n$.

Since \(F\) is strongly connected, Lemma~\ref{Gamma_B^n_for_SC} implies that \(\{U_1,\dots,U_k\}\) is also a partition of \(\Gamma_{E,B}^n\), and over each \(U_i\), the map \(s_i\) defines a section of \(\vv{\Pi}_n\). This completes the proof.
\end{proof}

Using Proposition~\ref{equality}, we now compute the sequential directed parametrized topological complexity of several $d$-fibrations.
%\begin{proof}
 %  Since \(F\) is strongly connected, Proposition~\ref{comparision: directed and usual} shows that it suffices to prove
%\[
%\vv{\TC}_n[p\colon E\to B]\leq \TC_n[p\colon E\to B].
%\]
%Suppose $\TC_n[p\colon E\to B]=k$. Then since $E$ and $B$ are ENRs, $E_B^n$ can be partitioned by ENRs $\{U_1,\dots,U_k\}$ such that over each $U_i$, there exists section $s_i$ of the evaluation map $\Pi_n\colon E_B^I\to E_B^n$. Note that for any $(e_1,\dots,e_n)\in E_B^n$, $p(e_1)=\dots=p(e_n)=b$ say. Then $s_i(e_1,\dots,e_n)\in F_b$. Since the induced $d$-structure on $F$ is given by free path space, $s_i(e_1,\dots,e_n)\in dE_B$. Since $F$ is strongly connected, it follows from Lemma~\ref{Gamma_B^n_for_SC} that $\Gamma_{E,B}^n$ partioned by $\{U_i\}_{i=1}^k$ such that over each $U_i$, $s_i$ is a section of $\vv{\Pi}_n$. This completes the proof.
%\end{proof}
%-----------------------------------------------------------------------------
\begin{example} \label{directed Fadell--Neuwirth}
 \normalfont{The problem of collision-free motion planning for multiple robots in the presence of obstacles with unknown locations was investigated in \cite{C-F-W,PTCcolfree}. This problem is naturally modeled by the \emph{Fadell--Neuwirth fibration}. Recall that, for a topological space $Y$, the configuration space of $n$ ordered distinct points in $Y$ is
\[
F(Y,n)=\{(x_1,\ldots,x_n)\in Y^n \mid x_i\neq x_j \text{ whenever } i\neq j\}.
\]
The Fadell--Neuwirth fibration is the projection
\[
p\colon F(\mathbb{R}^k,m+n)\longrightarrow F(\mathbb{R}^k,m),\qquad
p(z_1,\ldots,z_{m+n})=(z_1,\ldots,z_m).
\]
By equipping $\mathbb{R}^k$ with its standard directed structure \cite[p.~53]{MG2009}, the induced directed structures on the configuration spaces make the map $p$ into a d-fibration.
%Let $F(\mathbb{R}^k,n)$ denote the configuration space of $n$ ordered distinct points in $\mathbb{R}^k$. Consider the Fadell--Neuwirth fibration \[p\colon F(\mathbb{R}^k,m+n)\to F(\mathbb{R}^k,m), \qquadp(z_1,\dots,z_{m+n})=(z_1,\dots,z_m).\]

%We equip $\mathbb{R}^k$ with its standard directed structure as in \cite[Page 53]{grandis2009directed}, and endow the configuration spaces with the induced directed structures. With these choices, $p$ becomes a fibration in the category of directed spaces.

By Proposition~\ref{equality} together with \cite[Theorem 8.1]{Farber-Paul1} and \cite[Theorem 1.3]{FarberPaul2}, we obtain that for any $n \geq 1$ and $m \geq 2$,
\[
\vv{\TC}_n\!\big[p\colon F(\mathbb{R}^k,m+n)\to F(\mathbb{R}^k,m)\big]
=
\begin{cases}
nm + m & \text{if } k \geq 3 \text{ is odd},\\[4pt]
nm + m - 1 & \text{if } k \geq 2 \text{ is even}.
\end{cases}
\]}
\end{example}
\begin{remark}\label{rmk}
Let \(p\colon E\to B\) be a principal \(G\)-bundle, where \((B,dB)\) is a directed space, and let \((E,dE)\) be equipped with the directed structure defined above. Then Proposition~\ref{equality}, together with \cite[Proposition 3.3]{Farber-Paul1}, yields
\[
\vv{\TC}_n[p\colon E\to B]= \TC_n(G).
\]
\end{remark}
\begin{example} \label{Hopf fibration}
\normalfont{Let $p_1 \colon \mathbb{S}^3 \to \mathbb{S}^2$ and $p_2 \colon \mathbb{S}^7 \to \mathbb{S}^4$ denote the complex and quaternionic Hopf fibrations, which are principal $\mathbb{S}^1$- and $SU(2)$-bundles, respectively.

Equip each sphere $\mathbb{S}^k$ any directed structure. For instance one may take the directed structure on $\mathbb{S}^k$ induced from the directed cube $\vv{I}^k$ via the quotient map $\vv{I}^k \to \mathbb{S}^k = \vv{I}^k/\partial\vv{I}^k$ (see \cite[Definition 8]{EGMFAS2020}). 
Endowing $\mathbb{S}^3$ and $\mathbb{S}^7$ with these induced directed structures, Remark~\ref{rmk} together with the fact that $\TC_n(\mathbb{S}^1)=n$ and $\TC_n(\mathbb{S}^3)=n$ \cite{RUD2010}, yields:
\begin{enumerate}
    \item $\vv{\TC}_n[p_1 \colon \mathbb{S}^3 \to \mathbb{S}^2] = \TC_n(\mathbb{S}^1) = n$.
    \item $\vv{\TC}_n[p_2 \colon \mathbb{S}^7 \to \mathbb{S}^4] = \TC_n(SU(2)) = \TC_n(\mathbb{S}^3) = n$.
\end{enumerate}}
\end{example}
\begin{example}
\normalfont{Consider the Hopf fibration $p\colon \mathbb{S}^{2n+1}\longrightarrow \mathbb{C}P^n.$
Equip $\mathbb{C}P^n$ with any directed structure. In particular,
since $\mathbb{C}P^n$ can be realized as a directed mapping cone, \cite{grandis2002directed} provides a family of directed structures on $\mathbb{C}P^n$. By \cite[Definition~5.1]{SDNDAS2025}, the directed structure on $\mathbb{C}P^n$ induces a directed structure on $\mathbb{S}^{2n+1}$ for which $p$ is a directed map. Since $p$ is a principal $\mathbb{S}^1$-bundle and $\TC_n(\mathbb{S}^1)=n$ \cite{RUD2010}, it follows from Remark~\ref{rmk} that
$\vv{\TC}_n[p\colon \mathbb{S}^{2n+1}\to \mathbb{C}P^n]
=
\TC_n(\mathbb{S}^1)
=
n$.}
\end{example}

\subsection*{Acknowledgments.}
The authors would like to thank Prof. Mark Grant for his valuable insights on the computation of the sequential directed topological complexity of the directed circle. Navnath Daundkar gratefully acknowledges the support of DST–INSPIRE Faculty Fellowship (Faculty Registration No. IFA24-MA218), as well as Industrial Consultancy and Sponsored Research (IC\&SR), Indian Institute of Technology Madras for the New Faculty Initiation Grant (RF25261395MANFIG009294). Abhishek Sarkar expresses his deepest gratitude to the Indian Institute of Science Education and Research Pune for its support during his academic visit, where this work began. Ankur Sarkar gratefully acknowledges Prof. Sankaran Viswanath for providing him with a visiting position at The Institute of Mathematical Sciences, during which a major part of this work was carried out.
%%%%%%%%%%%%%%%%%%%%%%%%%%%%%%%%%%%%%%%%%%%%%%%%%%%%%%%%%%%%%%%%%%%%%%%%%
\bibliographystyle{plain} 
\bibliography{references}

\end{document}